\documentclass{article}
\usepackage{amsmath,amssymb,amsthm,graphicx}
\usepackage{hyperref}
\pdfoutput=1

\newtheorem{theorem}{Theorem}[section]
\newtheorem{lemma}{Lemma}[section]
\newtheorem{proposition}{Proposition}[section]

\renewcommand{\Pr}{P}

\title{The Search for Truth through Data: \\
{\large NP Decision Processes, ROC Functions, $P$-Functionals, Knowledge Updating and Sequential Learning}}

\author{Edsel A.\ Pe\~na \\ 
Department of Statistics \\
University of South Carolina \\
Columbia, SC 29208 USA}

\date{\today}

\begin{document}

\maketitle

\begin{abstract}

Strong and vehement criticisms of statistical decision-making procedures, especially those using $P$-values and a level of significance (LoS) of $\alpha = .05$, have abounded in recent years and is still continuing. For the sake of the scientific endeavor and the search for truth through data, there is tremendous impetus to re-examine and to improve these statistical decision-making procedures. This paper re-visits the fundamental problem of deciding the truth, based on data, between two competing hypotheses: a null $H_0$ and an alternative $H_1$. 
The Neyman-Pearson (NP) most powerful (MP) decision function, together with its power, remains the linchpin of  statistical decision-making. Associated with this procedure is a decision process and a receiver operating characteristic (ROC) function. It is proposed that in reporting the outcome of the NP MP decision function, for a specified LoS, that it should always be accompanied by the value of the (equivalently, logarithm) likelihood ratio based on the decision function. The $P$-functional, which is the usual $P$-value statistic, associated with the decision process is re-examined. It is pointed out that $P$ could be used in an equivalent implementation of the NP MP decision function, but if one wants to use its value to quantify the magnitude of support for either $H_0$ or $H_1$, then it should be the value of its (equivalently, logarithm) density function under $H_1$, which is the derivative of the ROC function, which  should be reported. This will avoid the fallacy that smaller values of $P$ are more supportive of $H_1$. Replicability of results is discussed in the context of realizations of the decision function or the $P$-functional. It is demonstrated that a coherent manner of acquiring knowledge about $H_0$ and $H_1$ is via sequential learning through Bayesian updating. But, it is also shown that publication bias could lead sequential updating astray in determining which of $H_0$ or $H_1$ is true. It is argued that a decision-maker can choose his {\em own} LoS, instead of using the conventional LoS of $\alpha = .05$, since the summary measures accompanying the realized decision take into account the chosen LoS. Since a decision-maker is then free to choose his LoS, the question of how to choose it optimally arises. Three approaches for choosing the LoS are discussed, each appropriate for a specific situation a decision-maker encounters. A new approach to sample size determination is also described.  The ideas are illustrated using concrete problems and a re-examination of Fisher's lady tea-tasting experiment. It is hoped that recommendations under this fundamental setting will extend to complex settings and lessen criticisms of methods relying on $P$-values and an LoS of $\alpha=.05$.

\medskip


\medskip


\end{abstract}

\section{The Search for Truth and Knowledge Representation}
\label{sect-introduction}

Apart from humankind's desire to propagate its own species or, as Richard Dawkins argues in \cite{Daw06}, their own genes, the search for truth and the acquisition of knowledge is one of the noblest among the multitude of human endeavors, if not its {\em raison d'\^etre}. As John 8:32 in the Bible states:
%
{``And ye shall know the truth and the truth shall make you free.''}
%
This passage, etched in a wall of the main lobby of the headquarters of the United States Central Intelligence Agency, should not just be the mantra of intelligence people, but it should be for all scientists trying to comprehend nature and the workings of the universe.

There are several types of truths. There are mathematical truths, called theorems, which are established through a purely deductive process using logic and starting from basic axioms and rules of operations. An example of such a mathematical truth is that $\sqrt{2}$ is not a rational number, a fact that is established with certainty mathematically. But there are truths, which usually correspond to physical realities, that could not be established through purely deductive reasoning, or even if they are derived through mathematical arguments, still require for their confirmations an inductive statistical approach using data obtained from observations and/or designed experiments. Professor Bradley Efron, the 2019 International Prize in Statistics awardee, refers to these type of truths as ``eternal and long-lasting'' during his speech at the World Statistics Congress this year in Kuala Lumpur, Malaysia \cite{EfronWSC2019}. These could include 
physical truths arising from theoretical and mathematical considerations. An example is Albert Einstein's result concerning the bending of light rays due to massive bodies, a consequence of his general theory of relativity, but which still required verification based on observed data to be fully accepted by the scientific community as a truth about reality. Incidentally, this year, 2019, is the centenary of the empirical verification of this theoretical result of Einstein by the team headed by Sir Arthur Eddington based on observations about the bending of light rays during a total solar eclipse; see \cite{NYT1919}. There are also truths, which may be subject to different interpretations, such as pertaining to whether eating red and processed meat poses health risks to people, an issue that flared up anew with the recent publication of an article contradicting purportedly established results showing the hazards of eating red and processed meat (see \cite{Zer19,Car19}).

But how do we characterize or describe truths? First, truth could be in the form of a set of mathematical formulas, a chemical structure, a biological pathway, or specified by a deterministic quantity, though the value of the quantity may not be known. Examples of these are a follows.
%
(a) Physical laws that are encapsulated in mathematical formulas such as Isaac Newton's $F = ma$, Einstein's $E = mc^2$, James Clerk Maxwell's equations governing electricity and magnetism, and Paul Dirac's relativistic wave equation.
(b) The double-helix structure of deoxyribonucleic acid (DNA), the carrier of genetic or hereditary information in almost all living things.
(c) The speed of light in a vacuum, denoted by $c$.
(d) The state of a defendant in a criminal murder trial which could either be guilty or innocent, but not both.
(e) The proportion ($\theta$) of the American Electorate who will vote for Donald Trump on November 3, 2020, Presidential Election Day. Here, truth may pertain to whether $\theta \ge .50$ or $\theta < .50$.
(f) The exact amount of federal taxes paid by Trump during the period from 2005 to 2015.
%
Second, truth could be represented via probability distributions. Some examples falling in this category are as follows.
%
(a) The number of deaths ($N$) that will be caused by vehicular accidents in 2020 in the State of South Carolina. A possible representation of the truth about $N$ is that it has a Poisson probability function with mean rate $\lambda$.
(b) The number of mass shootings that will occur in the United States during 2020. Again, this could be modeled by a Poisson probability function with mean rate $\lambda$.
(c) The position and momentum of a sub-atomic particle in a chamber, which according to Heisenberg's Uncertainty Principle, could only be described by their joint probability function.
(d) The temperature at noontime on January 1, 2020 at Columbia Metropolitan Airport, which could be described by a probability distribution function, such as a Gaussian distribution with mean $\mu$ and standard deviation $\sigma$.
%

However, even if the truth could be represented in a deterministic manner, when data are gathered through experiments for use in verifying or confirming the hypothesized deterministic representation, such data will be contaminated by measurement errors or random noise, so the observables are still described by a probability distribution which contains information about the underlying truth. For example, when Albert A.\ Michelson performed his experiment in 1879 to determine the speed of light, $c$, his measured values were not all identical since they were a function of $c$ and some contaminating values.

By virtue of the inductive statistical approach in determining, establishing, or confirming a truth, total certainty may not be achievable, in contrast to those obtained through a purely deductive approach such as the truth about the irrationality of $\sqrt{2}$. To illustrate via a very simple example, consider possessing a `coin' which is either ordinary (has a head and a tail) or mutated (has two heads).  Suppose that we could not directly examine it, but we could observe the outcome (face up) when it is tossed. We could keep tossing this coin sequentially and if after $n$ tosses we have obtained $n$ heads, then we still could not be certain that it is the mutated coin, however large $n$ is. But, an observed tail in any of these $n$ tosses implies with certainty that we have the ordinary coin.

In this paper we revisit the problem of determining which of two simple hypotheses is true based on data. This setting corresponds to the Neyman-Pearson framework of decision-making in its most fundamental form. We will then discuss the notion of a $P$-functional, the $P$-value statistic, which arguably is the concept most identified with statistical thinking, but which has gained notoriety in recent years, and has been attacked and blamed for the ills that have befallen the scientific endeavor; see, for instance, the papers which deal with issues and controversies about $P$-values \cite{HalKra02,TraMar15,Ion05,GigMar15,Nuz15,Bak16,WasLaz16,WasSchLaz19,Gre19,McSGalGelRobTac19}. In the same token that a deeper understanding of hypothesis testing was facilitated by first considering the simple null hypothesis versus simple alternative hypothesis, leading to the Neyman-Pearson Fundamental Lemma \cite{NeyPea33}, it is our viewpoint that a better understanding of $P$-functionals will also be facilitated by considering this fundamental setting. We will then examine ways in which we could improve the process of determining the truth based on data, address how results of statistical decision-making and observed $P$-functionals should be reported to be more informative and less misleading. We will also discuss the issue of replicability of scientific results, a concept in the forefront of what is considered good science. In the process, we will discuss the idea of sequential learning as a coherent way of getting to the truth. However, we will also  demonstrate the perils of publication bias in the sequential acquisition of knowledge. A weakness of the current practice of statistical decision-making is the reliance on a cut-off value for the level of significance (LoS), usually set at 5\%, but which is difficult to justify logically within the mathematical theory. We will therefore examine the issue of determining an optimal LoS which will lead to a new approach to sample size determination. The ideas and proposals will be illustrated using concrete situations, including a re-examination of R.\ A.\ Fisher's lady tea-tasting experiment, which ushered null hypothesis significance testing (NHST).

\section{Experiments, Observables, and Models}
\label{sec-experiments}

We will be interested in the simplest of all decision problems: determining which of two competing simple hypotheses, denoted by $H_0$ and $H_1$, is true. To do this determination, an experiment or study will be performed leading to observing the realization $x$ of a random entity $X$ taking values in a sample space $\mathfrak{X}$. In parallel, we also perform another experiment leading to observing the realization $u$ of a standard uniform random variable $U$ taking values in $\mathfrak{U} = [0,1]$. $X$ and $U$ are stochastically independent. Thus, the data will be $(x,u) \in \mathfrak{X} \times \mathfrak{U}$.
When $H_0$ is the truth, $X$ is governed by a probability measure $P_0$, whereas if $H_1$ is the truth, then $X$ is governed by a probability measure $P_1$, where $P_0$ and $P_1$ are probability measures on the measurable space $(\mathfrak{X}, \sigma(\mathfrak{X}))$. We denote by $\nu$ a measure on $(\mathfrak{X}, \sigma(\mathfrak{X}))$ which dominates $P_0$ and $P_1$ (e.g., $\nu = P_0 + P_1$), and denote by $f_0$ and $f_1$ the probability density functions (pdfs) associated with $P_0$ and $P_1$ with respect to $\nu$.
Thus, we may re-state the decision problem as deciding between $H_0: X \sim f_0$ or $H_1: X \sim f_1$ based on the data $(x,u)$. Note that, to be more precise, under $H_i$, the distribution of $(X,U)$ is $P_i \otimes \lambda$, where $\lambda$ is Lebesgue measure on $(\mathfrak{U},\sigma(\mathfrak{U}))$, hence it has density function $f_i^*(x,u) = f_i(x) I\{0 \le u \le 1\}$ with respect to $\nu \otimes \lambda$.
The introduction of the parallel experiment leading to observing $U$ is to provide a randomizer in case there is a need to randomize to make the final decision. We generate this randomizer $u$ in conjunction with the experiment leading to $x$ so that the use of the randomizer, if needed, does {\em not} acquire an outsize importance in the decision-making process. Thus, $u$ could be viewed as just being a portion of the data $(x,u)$. For a discussion of the use of randomizers in decision-making and to address criticisms of  their use, see section 3 of \cite{HabPen11}.

There is a dichotomy of what we would like to do regarding $H_0$ and $H_1$. First, given the observed data $(x,u)$, we may want to make a decision which of $H_0$ or $H_1$ is the truth, with proper consideration of the costs associated with erroneous choices. This is the purview of decision theory and the usual hypothesis testing approach. Second, upon seeing the data $(x,u)$, we may simply want to update our current knowledge about $H_0$ and $H_1$. A question arises how to represent our knowledge of which of $H_0$ and $H_1$ is the truth. We shall do so by assigning subjective prior probabilities of $\kappa_0$ and $\kappa_1 = 1 - \kappa_0$, with $\kappa_0 \in (0,1)$ representing our prior and current belief that $H_0$ is the truth. Given the observed data, the result of the decision function, or the value of the $P$-functional, the prior probabilities will then be updated, via Bayes theorem, to the posterior probabilities, which will then serve as, or could inform, the next set of prior probabilities of $H_0$ and $H_1$ to be used in the next study. This is the essence of sequential learning and updating. The first and second goals could be integrated by pre-specifying a threshold such that when the posterior probability of $H_0$ [$H_1$] surpasses this threshold, we then declare that $H_0$ [$H_1$] corresponds to the truth.

\section{Space of Decision Functions}
\label{sec-decision functions}

Let $\mathfrak{G}$ be  the space [of equivalence classes] of all measurable functions from $(\mathfrak{X} \times \mathfrak{U}, \sigma(\mathfrak{X}) \otimes \sigma(\mathfrak{U}))$ into $(\Re,\sigma(\Re))$ which are square-integrable with respect to $P_0 \otimes \lambda$ or $f_0 d(\nu \otimes \lambda)$. We endow this space with the inner product given by, for every $g_1, g_2 \in \mathfrak{G}$,
$$\langle g_1, g_2 \rangle = \int_{\mathfrak{X} \times \mathfrak{U}} g_1 g_2 f_0 d(\nu \otimes \lambda) = \int_0^1 \int_\mathfrak{X} g_1(x,u) g_2(x,u) f_0(x) \nu(dx) du.$$
The norm for $\mathfrak{G}$ is $||g|| = \sqrt{\langle g,g \rangle}$. The space of decision functions for $H_0$ versus $H_1$ is the subset $\mathfrak{D}$ of $\mathfrak{G}$ given by
\begin{equation}
\label{space of decision functions}
\mathfrak{D} = \{\delta \in \mathfrak{G}: \forall (x,u) \in \mathfrak{X} \times \mathfrak{U}, \delta(x,u) \in \{0,1\}\}.
\end{equation}
For $\delta \in \mathfrak{D}$, $\delta(x,u) = 1 (0)$ means to decide in favor of $H_1$ ($H_0$).

The likelihood function associated with $f_0$ and $f_1$ will be defined via
%
$\Lambda(x) = f_1(x)/f_0(x)$
%
with the convention that $0/0 = 0$.
Observe that in terms of the expectation operators, we have
$$E_{0} [g(X,U)] = \langle g, 1 \rangle \quad \mbox{and} \quad E_{1} [g(X,U)] = \langle g, \Lambda \rangle$$
with $E_{i}(\cdot)$ the expectation operator under $P_i \otimes \lambda$ for $i = 0, 1$.

Given a $\delta \in \mathfrak{D}$, its size and power are defined, respectively, via
\begin{equation}
\label{size and power}
\alpha_\delta = \langle \delta, 1 \rangle \quad \mbox{and} \quad \pi_\delta = \langle \delta, \Lambda \rangle.
\end{equation}
For $\alpha \in [0,1]$, a decision function $\delta \in \mathfrak{D}$ is said to be of level $\alpha$ if $\alpha_\delta \le \alpha$. A $\delta^* \in \mathfrak{D}$ is a most powerful  (MP) decision function of level $\alpha$ if $\alpha_{\delta^*} \le \alpha$ and for any other $\delta \in \mathfrak{D}$ with $\alpha_\delta \le \alpha$, we have $\pi_{\delta^*} \ge \pi_\delta$. We write such an MP decision function of level $\alpha$ as $\delta^*(\alpha) \equiv \delta^*(\cdot,\cdot;\alpha)$.
The alternative formulation using inner products immediately indicates that if one wants to maximize the power $\pi_{\delta^*(\alpha)} \ge \pi_\delta$ among all $\{0,1\}$-valued functions $\delta$, then the desired $\delta$ should be 1 (0) when $\Lambda$ is large (small), since the maximization problem in Lagrange form is equivalent to maximizing the mapping 
$$(\delta,\eta) \mapsto \langle \delta, \Lambda \rangle - \eta(\langle \delta, 1\rangle - \alpha) =
\langle \delta, (\Lambda - \eta) \rangle + \eta\alpha.$$
In this form we see that to maximize with respect to $\delta$, we must take $\delta^*(x,u) = 1 (0)$ whenever $\Lambda(x) > (<) \eta$ for some $\eta$.
This is the content of the Fundamental Lemma of Neyman and Pearson (1933) stated below. See also \cite{LehRom05,Sch95}.

\begin{theorem}{[Neyman-Pearson Lemma (1933) \cite{NeyPea33}]}
For $\alpha \in (0,1)$, a necessary and sufficient condition for $\delta^*(\alpha)$ to be an MP decision function of level $\alpha$ for $H_0: f = f_0$ versus $H_1: f = f_1$ is that, for all $(x,u) \in \mathfrak{X} \times \mathfrak{U}$, 
$$\delta^*(x,u;\alpha) = \left\{
\begin{array}{ccc}
1  & \mbox{if} & \Lambda(x) > c(\alpha) \\
0 & \mbox{if} & \Lambda(x) < c(\alpha)
\end{array}
\right.$$ 
for some $c(\alpha) \ge 0$ with $\alpha_{\delta^*(\alpha)} = \alpha$. A particular choice of this MP decision rule is
\begin{equation}
\label{NP decision rule}
\delta^*(x,u;\alpha) = I\{\Lambda(x) > c(\alpha)\} + I\{\Lambda(x) = c(\alpha); u \le \gamma(\alpha)\}
\end{equation}
where, with $\chi_c(x) = I\{\Lambda(x) > c\}$ and $\chi_{c-}(x) = I\{\Lambda(x) \ge c\}$,
$$c(\alpha) = \inf\{c \ge 0: \langle \chi_c, 1 \rangle \le \alpha\} \quad \mbox{and} \quad
\gamma(\alpha) = \frac{\alpha - \langle \chi_{c(\alpha)}, 1 \rangle}{\langle \chi_{c(\alpha)-}, 1 \rangle - \langle \chi_{c(\alpha)}, 1 \rangle}.$$
\end{theorem}

\begin{proof}
For the $\delta^*$ in the statement of the theorem, and for any other $\delta$ with $\langle \delta, 1 \rangle \le \alpha$, $(\delta^* -\delta)(\Lambda - c(\alpha)) \ge 0$. Therefore,
$\langle \delta^* -\delta, \Lambda - c(\alpha) \rangle \ge 0.$
Expanding and re-arranging terms, we obtain that
$\langle \delta^*, \Lambda \rangle - \langle \delta, \Lambda \rangle \ge c(\alpha) (\langle \delta^*,1 \rangle - \langle \delta, 1 \rangle),$
and since $c(\alpha) \ge 0$ and $\langle \delta^*,1 \rangle = \alpha$ and $ \langle \delta, 1 \rangle \le \alpha$, then the right-hand side is at least equal to 0. Thus, $\pi_{\delta^*} = \langle \delta^*, \Lambda \rangle \ge \langle \delta, \Lambda \rangle = \pi_{\delta}$. The particular form in (\ref{NP decision rule}) clearly satisfies the size condition $\alpha_{\delta^*} =  \langle \delta^*, 1 \rangle = \alpha$.
\end{proof}

Most often we are able to simplify this MP decision function when we could express the likelihood function via
$\Lambda(x) \equiv f_1(x)/f_0(x) = Q[S(x)]$
for some statistic $S(\cdot)$ taking values in a measurable space $(\mathfrak{S},\sigma(\mathfrak{S}))$ and some measurable mapping $Q$ from $(\mathfrak{S},\sigma(\mathfrak{S}))$ into $(\Re_+,\mathfrak{B}_+)$, where $\mathfrak{B}_+$ is the Borel sigma-field of subsets of $\Re_+$. The set $\{x \in \mathfrak{X}: \ \Lambda(x) > c\}$ is equal to $\{x \in \mathfrak{X}:\  S(x) \in Q^{-1}(c,\infty)\}$. If $Q(\cdot)$ is nondecreasing in $S(\cdot)$, then $\{x \in \mathfrak{X}: \Lambda(x) > c\} = \{x \in \mathfrak{X}: S(x) > d\}$ for some $d$. Obtaining $d(\alpha)$ and $\gamma(\alpha)$ then becomes simpler since they are obtained under the $H_0: f=f_0$ distribution of $S(X)$, which might be easier to obtain compared to the distribution of $\Lambda(X)$.
The power of the level-$\alpha$ MP decision rule $\delta^*(\alpha)$ in (\ref{NP decision rule}) is given by
$$\rho(\alpha) \equiv \pi_{\delta^*(\alpha)} = \langle \delta^*(\alpha), \Lambda \rangle = \langle \chi_{c(\alpha)}, \Lambda \rangle +
\gamma(\alpha) \langle \chi_{c(\alpha)-} - \chi_{c(\alpha)}, \Lambda \rangle.$$
%

\section{NP Decision Process and ROC Function}
\label{sec-NP processes}

From the MP decision function given in (\ref{NP decision rule}), we could form a stochastic process, which will be called an NP decision process, given by
\begin{equation}
\label{NP Decision Process}
\Delta = \{\delta^*(\alpha) \equiv \delta^*(\cdot,\cdot;\alpha): \alpha \in [0,1]\}.
\end{equation}
This process has right-continuous non-decreasing sample paths with state space $\{0,1\}$. Associated with this decision process is the receiver operating characteristic (ROC) function given by
\begin{equation}
\label{ROC Function}
\rho = \{\rho(\alpha): \alpha \in [0,1]\}.
\end{equation}
We present properties of the ROC function in the following proposition. Some of these results will be needed in establishing Theorem \ref{posterior convergence}, a result pertaining to sequential learning.


\begin{proposition}
\label{prop-ROC properties}
Assume that $\nu\{x \in \mathfrak{X}: f_0(x) \ne f_1(x)\} > 0$ and consider the ROC function $\alpha \mapsto \rho(\alpha)$ associated with the NP decision process $\Delta = \{\delta^*(\alpha): \alpha \in (0,1)\}$. Then
\begin{itemize}
\item[(a)] $\forall \alpha \in (0,1): \rho(\alpha) > \alpha$;
\item[(b)] $\alpha \mapsto \rho(\alpha)$ is non-decreasing, concave, continuous, and is piecewise differentiable with $\alpha \mapsto \rho^\prime(\alpha) = \frac{d}{d\alpha} \rho(\alpha)$ non-increasing, almost everywhere (wrt Lebesgue measure);
\item[(c)] $\lim_{\alpha \rightarrow 0} \rho(\alpha) = 0$ and $\lim_{\alpha \rightarrow 1} \rho(\alpha) = 1$;
\item[(d)] $\rho^{\prime}(0+) \equiv \lim_{\alpha \downarrow 0} \frac{\rho(\alpha)}{\alpha} > 1$ and $\rho^{\prime}(1-) \equiv \lim_{\alpha \uparrow 1} \frac{1 - \rho(\alpha)}{1 - \alpha} < 1$.
\end{itemize}
\end{proposition}

\begin{proof}
Result (a) follows by taking the decision function $\delta(x;\alpha) = \alpha$ and using the condition that $\nu\{f_0 \ne f_1\} > 0$. For the results in (b), given $\alpha_1 < \alpha_2$, the collection of decision functions of level $\alpha_1$ is contained in the collection of decision functions of level $\alpha_2$, hence the MP decision function in the latter collection will have power at least equal to the MP decision function in the former collection. Thus, $\rho(\alpha_1) \le \rho(\alpha_2)$. Now, let $\alpha_1, \alpha_2 \in (0,1)$ and denote by $\delta_1^*$ and $\delta_2^*$ the MP decision functions of size $\alpha_1$ and $\alpha_2$. For $a \in (0,1)$, let $\alpha = a\alpha_1 + (1-a)\alpha_2$ and denote by $\delta^*$ the MP decision function of size $\alpha$.  Consider the decision function $\bar{\delta} = a \delta_1^* + (1-a)\delta_2^*$. The size of $\bar{\delta}$ is $a\alpha_1 + (1-a)\alpha_2$ and its power is $a\rho(\alpha_1) +  (1-a)\rho(\alpha_2)$.  Since $\delta^*$ is the MP decision function at size $\alpha = a\alpha_1 + (1-a)\alpha_2$, and $\bar{\delta}$ is also a decision function of size $\alpha$, then the power of $\delta^*$ is at least that of $\bar{\delta}$. Therefore,
$$\rho(a \alpha_1 + (1-a)\alpha_2) \ge  a\rho(\alpha_1) + (1-a)\rho(\alpha_2).$$
Since $\alpha_1$, $\alpha_2$, and $a$ are arbitrary elements of $(0,1)$, then this shows that $\alpha \mapsto \rho(\alpha)$ is concave. Continuity follows from this concavity. Concavity also implies that the left-hand and right-hand derivatives of $\rho(\alpha)$ exist, and since it could have at most a countable number of points where left-hand and right-hand derivatives are not equal, then it is piecewise differentiable.
Second result in (c) follows immediately from $1 \ge \rho(\alpha) > \alpha$ for all $\alpha \in (0,1)$.  Let $\bar{\Lambda} = \sup\{b \in \Re: \Pr_0\{\Lambda > b\} > 0\}$. Then $c(\alpha) \rightarrow \bar{\Lambda}$ as $\alpha \rightarrow 0$, hence 
\begin{eqnarray*}
\lim_{\alpha \rightarrow 0} \rho(\alpha) & = & \lim_{\alpha \rightarrow \infty} \left\{P_1\{\Lambda > c(\alpha)\} + \left[\frac{\alpha - \Pr_0(\Lambda > c(\alpha))}{\Pr_0(\Lambda = c(\alpha))} \right] \Pr_1(\Lambda = c(\alpha))\right\} \\
& = & 0
\end{eqnarray*}
since $\alpha \ge \alpha - \Pr_0(\Lambda > c(\alpha)) \ge 0$ and 
$$\Pr_1(\Lambda > c(\alpha)) = \int I\{\Lambda(x) > c(\alpha), f_0(x) >0\} \Lambda(x) f_0(x) \nu(dx) \rightarrow 0$$ 
by virtue of $\alpha \ge \Pr_0\{\Lambda > c(\alpha)\} = \int I\{f_0(x) > 0, \Lambda(x) > c(\alpha)\} f_0(x) \nu(dx) \rightarrow 0$ so that $\nu\{x: f_0(x) > 0, \Lambda(x) > c(\alpha)\} \rightarrow 0$ and using Dominated Convergence Theorem.
For the results in (d), as $\alpha \downarrow 0$, $c(\alpha) \rightarrow \bar{\Lambda}$. If the distribution of $\Lambda$ under $P_0$ is continuous, then
$$\rho(\alpha) = \Pr_1\{\Lambda > c(\alpha)\} = \int I\{\Lambda > c(\alpha)\} \Lambda f_0 \ge c(\alpha) \int I\{\Lambda > c(\alpha)\} f_0 = c(\alpha) \alpha,$$
hence $\rho(\alpha)/\alpha \ge c(\alpha) \rightarrow \bar{\Lambda}$ as $\alpha \downarrow 0$. But since $\nu\{f_0 \ne f_1\} > 0$ and $\int f_0 = \int f_1 = 1$, then $\bar{\Lambda} > 1$; in fact, it is possible to have $\bar{\Lambda} = \infty$. If $\Pr_0\{\Lambda = \bar{\Lambda}\} > 0$, then for $\alpha < \Pr_0\{\Lambda = \bar{\Lambda}\}$, $c(\alpha) = \bar{\Lambda}$, hence $\Pr_0\{\Lambda > c(\alpha)\} = 0$. For such an $\alpha$, we then have
$$\rho(\alpha) = \alpha \frac{\Pr_1\{\Lambda = \bar{\Lambda}\}}{\Pr_0\{\Lambda = \bar{\Lambda}\}} = \alpha \bar{\Lambda}.$$
Therefore, $\rho(\alpha)/\alpha = \bar{\Lambda}$ for all $\alpha < \Pr_0\{\Lambda = \bar{\Lambda}\}$, hence $\rho(\alpha)/\alpha \rightarrow \bar{\Lambda} > 1$ as $\alpha \downarrow 0$.
Let $\underline{\Lambda} = \in\{b \in \Re: \Pr_0\{\Lambda \le b\}  > 0\}$. Then, when $\alpha \uparrow 1$, $c(\alpha) \rightarrow \underline{\Lambda}$. If the distribution of $\Lambda$ under $P_0$ is continuous, then
$$1 - \rho(\alpha) = \Pr_1\{\Lambda \le c(\alpha)\} = \int I\{\Lambda \le c(\alpha)\} \Lambda f_0 \le c(\alpha) \Pr_0\{\Lambda \le c(\alpha)\} = c(\alpha) (1-\alpha)$$
hence $[1-\rho(\alpha)]/[1-\alpha] \le c(\alpha) \rightarrow \underline{\Lambda}$ as $\alpha \uparrow 1$. By same argument as for $\bar{\Lambda}$, $\underline{\Lambda} < 1$, and could take the value 0. If $\Pr_0\{\Lambda = \underline{\Lambda}\} > 0$, then for $\alpha$ satisfying $1 - \alpha < \Pr_0\{\Lambda = \underline{\Lambda}\}$, we have $c(\alpha) = \underline{\Lambda}$. Therefore, for $\alpha > 1 - \Pr_0\{\Lambda = \underline{\Lambda}\} = \Pr_0\{\Lambda > \underline{\Lambda}\}$, we have
\begin{eqnarray*}
\rho(\alpha) & = & \Pr_1\{\Lambda > \underline{\Lambda}\} + 
\left[\frac{\alpha - (1 - \Pr_0\{\Lambda = \underline{\Lambda}\})}{\Pr_0\{\Lambda = \underline{\Lambda}\}} \right] \Pr_1\{\Lambda = \underline{\Lambda}\} \\ 
& = & \Pr_1\{\Lambda \ge \underline{\Lambda}\} - (1-\alpha) \left[\frac{\Pr_1\{\Lambda = \underline{\Lambda}\}}{\Pr_0\{\Lambda = \underline{\Lambda}\}}\right].
\end{eqnarray*}
Thus, for $\alpha > \Pr_0\{\Lambda > \underline{\Lambda}\}$, we have
$$1-\rho(\alpha) = \Pr_1\{\Lambda < \underline{\Lambda}\} + (1-\alpha) \frac{\Pr_1\{\Lambda = \underline{\Lambda}\}}{\Pr_0\{\Lambda = \underline{\Lambda}\}} = (1-\alpha) \frac{\Pr_1\{\Lambda = \underline{\Lambda}\}}{\Pr_0\{\Lambda = \underline{\Lambda}\}}$$
since $0 \le \Pr_1\{\Lambda < \underline{\Lambda}\} =  \int I\{\Lambda < \underline{\Lambda}\} \Lambda f_0 \le \underline{\Lambda} \Pr_0\{\Lambda < \underline{\Lambda}\} = 0$. Therefore, as $\alpha \uparrow 1$, we have
$$\frac{1-\rho(\alpha)}{1-\alpha} \rightarrow \frac{\Pr_1\{\Lambda = \underline{\Lambda}\}}{\Pr_0\{\Lambda = \underline{\Lambda}\} }= \underline{\Lambda} < 1.$$
%
\end{proof}

\section{$P$-Functionals and their Properties}
\label{sec-P Functionals}

For the NP decision process $\Delta$, we define a (random) functional via
\begin{equation}
\label{P Functional}
P(X,U) = \inf\{\alpha: \delta^*(X,U;\alpha) = 1\}.
\end{equation}
This is the so-called $P$-{\em value} statistic, but we would like to stay away from this misnomer since this quantity is truly a {\em statistic}, that is, a random variable depending only on the sample data, and it is a characteristic of the decision process $\Delta$.
In terms of the $P$-functional, the size-$\alpha$ MP decision function is equivalent to
$$\delta^*(x,u;\alpha) = I\{P(x,u) \le \alpha\}.$$
The usual approach to testing $H_0$ versus $H_1$ is to pre-specify a level of significance (LoS) $\alpha_0$, which has usually been conventionally set to 0.05, and to use a test whose size is no more than $\alpha_0$. As such, in order to gain the most power, the test $\delta^*(\cdot,\cdot;\alpha_0)$ is utilized.
We present two important distributional properties of the $P$-functional.

\begin{theorem}
Under $H_0$, $P(X,U)$ has a standard uniform distribution, hence its density function under $H_0$ is $h_0(w) = 1$ for $w \in [0,1]$.
\end{theorem}

\begin{proof}
For $w \in [0,1]$, we have
\begin{eqnarray*}
\Pr_{0}\{P(X,U) \le w\} & = & \Pr_{0}\{\delta^*(X,U;w) = 1\} 
 =  E_{0}[\delta^*(X,U;w)] = \langle \delta^*(w), 1 \rangle
 =  w.
\end{eqnarray*}
\end{proof}

\begin{theorem}
Under $H_1$, $P(X,U)$ has distribution $\rho(\cdot)$, hence its density function under $H_1$ is $h_1(w) = \rho^\prime(w)$.
\end{theorem}

\begin{proof}
For $w \in [0,1]$, we have
\begin{eqnarray*}
\Pr_{1}\{P(X,U) \le w\} & = & \Pr_{1}\{\delta^*(X,U;w) = 1\} 
 =  E_{1}[\delta^*(X,U;w)] \\ & = & E_{0} [\delta^*(X,U;w) \Lambda(X)] = \langle \delta^*(w), \Lambda \rangle
 =  \rho(w).
\end{eqnarray*}
\end{proof}

Note therefore that if one could {\em only} observe the $P$-functional $P(X,U)$, the MP decision function of $H_0$ versus $H_1$ of size $\alpha_0$ is, by the Neyman-Pearson Lemma, $\delta(P(x,u);\alpha_0) = I\{P(x,u) \le \alpha_0\}$ since $w \mapsto \rho^\prime(w)$ is a non-increasing function. This coincides with the MP decision function based on $x$.
Note further that the existence of the $P$-functional has nothing to do with any pre-specified LoS. The $P$-functional exists whenever there is a decision process $\Delta$. The $P$-functional and the pre-specified LoS only come together when a decision about $H_0$ and $H_1$ is to be made. One may think of the $P$-functional as a decision process-induced transformation of the original data $(X,U)$ into a random variable (a statistic) $P(X,U)$ that is uniformly distributed under $H_0$ and whose distribution under $H_1$ is stochastically smaller than a standard uniform distribution, since $\rho(w) > w$ for every $w \in (0,1)$ with $\rho(0) = 0$ and $\rho(1) = 1$. As such there is nothing mysterious about the $P$-functionals (albeit, $P$-values) -- contrary to the negative attention and notoriety it has garnered in recent years! See the recent articles on $P$-values \cite{WasSchLaz19,FriBurHanWoo19} and the other references cited earlier.

That the $P$-functional is a statistic was emphasized in Kuffner and Walker \cite{KufWal19}, which also demonstrated that the $P$-value statistic is in a one-to-one and onto (a bijection) relationship with the sufficient statistic in the case when the sufficient statistic is one-dimensional. This bijection property between the $P$-functional and the sufficient statistic need not hold, however, if the sufficient statistic has dimension at least 2. For example, if $H_0: \mu=\mu_0, \sigma^2 > 0$ and $H_1: \mu \ne \mu_0, \sigma^2 > 0$ and the data set to test $H_0$ versus $H_1$ is $X_1, X_2, \ldots, X_n \sim N(\mu,\sigma^2)$, the sufficient statistic is $T = (\sum X_i, \sum X_i^2)$, which is two-dimensional and which could not be recovered if given only the $P$-value statistic based on the $t$-test statistic $T = (\bar{X} -\mu_0)/(S/\sqrt{n})$ of $H_0$ and $H_1$.  Note that in this case, the $t$-test statistic is not sufficient for $(\mu,\sigma^2)$. We note, though, that the bijection property between the $P$-functional and the sufficient statistic may still hold in higher-dimensions if we instead consider the {\em invariantly} sufficient test statistic. For example, the $t$-test statistic is invariantly sufficient under the location-scale group of transformations.

\section{Issue of Replicability}
\label{sec-Replicability}

In  the recent United States National Academies of Sciences, Engineering, and Medicine report on reproducibility and replicability in science \cite{NASReport2019}, a distinction is made between the notions of reproducibility and replicability. According to this report, {\em reproducibility} means being able to obtain consistent computational results using the same input data, computational steps, methods, code, and conditions of analysis; whereas,
{\em replicability} means being able to obtain consistent results across studies aimed at answering the same scientific question, each of which has obtained its own data. The notion of replicability is arguably the more important one in the context of the integrity and viability of scientific results. Let us examine the notion of replicability in the context of using the MP decision function and the $P$-functional in deciding between $H_0$ and $H_1$.

Consider a scientist performing a study to decide between $H_0$ and $H_1$. Following existing decision-making procedures, he specifies an LoS to use, and assume he chooses the traditional value of $\alpha_1 = 0.05$. Let $(x_1,u_1)$ be his observed data, and let
$$d_1 = \delta(x_1,u_1;\alpha_1) \quad \mbox{and} \quad p_1 = P(x_1,u_1)$$
be the observed decision and the observed $P$-functional, respectively.
Suppose it turns out that $d_1 = 1$. Then the scientist will conclude that $H_1$ is true, instead of $H_0$, at LoS $\alpha_1 = .05$. What does this {\em really} mean? If $H_0$ is in reality true, using the decision function $\delta(\cdot,\cdot;\alpha_1)$, there is a 0.05 chance that $H_0$ will be declared false; so the observed $d_1 = 1$ is one of these false discovery realizations. On the other hand, if $H_1$ is in reality true, there is a probability of $\rho(\alpha_1)$ of concluding that this is indeed the case, and the realized decision $d_1 = 1$ is one of these correct decisions. But this observed decision of $d_1 = 1$ does not preclude either of the two possibilities of $H_0$ or $H_1$. If a second scientist performs a study with LoS of $\alpha_2 = .05$ to decide between $H_0$ and $H_1$, this scientist may get a decision $d_2 = 0$, which will be contradicting the first scientist's decision, and we might then say that the first scientist's result is not replicable. But if $H_0$ is the true hypothesis, then the result of the second scientist was highly likely (probability of 0.95 prior to performing the study); while if $H_1$ is true, this false negative result by the second scientist was also plausible if the power $\rho(\alpha_2)$ is not high, with this power determined by the `effect size' or the `distance' between $f_0$ and $f_1$. As such the issue of replicability is a difficult question to resolve when considering the results of two scientists.

Let us examine this further when there are {\em many} replications. Suppose that there are $M = 100$ scientists, labeled $m=1,2,\ldots,M,$ who will perform their own independent studies, each using an LoS of $\alpha_m = .05$, to decide between $H_0$ and $H_1$, with the $m$th scientist using the MP decision function $\delta_m^*(\alpha_m) = \delta_m^*(x_m,u_m;\alpha_m)$ whose power is $\rho_m(\alpha_m)$. If $H_0$ is true, about 5 of these scientists will obtain values of $d_m = 1$ for their $\delta^*_m(x_m,u_m;\alpha_0)$; while, if $H_1$ is true, depending on their $\rho_m(\alpha_m)$s, there will tend to be more than 5 of them with $d_m = 1$. Under both cases, the occurrences of $d_m=1$ will appear in a random order for these $M$ scientists. With many independent replications, we will better see which of $H_0$ or $H_1$ is more plausible. {\em But,} any other scientist, or even the same scientist, performing another study to decide between $H_0$ and $H_1$ may come up with a decision which disagrees with the totality of results of the $M$ scientists. As such, when focusing on the results of individual scientists the replicability issue will always arise under both $H_0$ or $H_1$. 

The same thing happens with the $P$-functional. Under $H_0$, $P(X,U)$ has a uniform distribution over $[0,1]$, so that any value between $[0,1]$ is not surprising under $H_0$. Under $H_1$, the distribution of $P(X,U)$ will be right-skewed, so smaller values will be more likely. But, with one realization of $P(X,U)$ that is smaller than $\alpha=0.05$, it again does {\em not} preclude either $H_0$ or $H_1$. With 100 scientists performing independent studies and observing their $P$-functionals, if $H_0$ is true, then their observed $P$'s will tend to be uniform over $[0,1]$; while if $H_1$ is true, then their observed $P$'s will tend to have a histogram that is right-skewed. 
We also reiterate that if the observed $P(x,u) = p$ is less than $\alpha =$0.05, it is faulty to state that $H_0$ is rejected at LoS of $p$. But, it is fine to conclude that $H_0$ is rejected at LoS of $\alpha = 0.05$, provided that this LoS value was chosen prior to observing the data $(x,u)$. The point is that the observed data should not determine the LoS that is used to make the decision.

Appropriately, the NAS Report \cite{NASReport2019} points out the difficulty in assessing the notion of replicable results. It states as follows:

\begin{quote}
Because of the complicated relationship between replicability and its variety of sources, the validity of scientific results should be considered in the context of an entire body of evidence, rather than an individual study or an individual replication. Moreover, replication may be a matter of degree, rather than a binary result of ``success'' or ``failure.''
\end{quote}

To concretely demonstrate these issues concerning replicability, consider the two-sample problem of testing hypotheses about the difference of two population means, with the pair  of hypotheses
$$H_0: \mu_1 - \mu_0 = 0 \quad \mbox{versus} \quad H_1: \mu_1 - \mu_0 = \kappa$$
where $\kappa > 0$. For the $m$th scientist, with $m \in \{1,2,\ldots,M=100\}$, we suppose that his sample data are realizations of
$$X_m = (X_{m1},X_{m2},\ldots,X_{mn_m}) \sim N(\mu_0,\sigma^2);$$
$$Y_m = (Y_{m1},Y_{m2},\ldots,Y_{mn_m}) \sim N(\mu_1,\sigma^2),$$
with $X_m$ independent of $Y_m$. Here $\sigma^2 > 0$ is assumed unknown. Let the sample sizes be determined according to $n_m \sim POI(\lambda) + 5$. The possibly differing sample sizes among these $M$ scientists is meant to model the realistic situation where they have different experimental resources available to each of them, so some will have larger sample sizes, hence will have higher powers for their decision functions. It is without doubt that there will be investigators which will have more resources, possibly due to more research funding, while others will have meager resources but would still want to contribute to the resolution of a scientific question {\em even} with their limited resources. For the $m$th scientist the summary statistics from the two samples will be the sample means $\bar{X}_m$ and $\bar{Y}_m$ and the sample variances $S_{mX}^2$ and $S_{mY}^2$. The $\alpha$-size test procedure to be used by the $m$th scientist is
$\delta_m^*(X_m,Y_m;\alpha) = I\{T_{mc} \ge t_{2(n_m-1);\alpha}\}$
where 
$$T_{mc} = \frac{\bar{Y}_m - \bar{X}_m}{S_m \sqrt{2/n_m}}; \quad
S_m^2 = \frac{S_{mX}^2 + S_{mY}^2}{2}; \quad
t_{2(n_m-1);\alpha} = \mathcal{T}^{-1}(1-\alpha;2(n_m-1)),$$
where $\mathcal{T}(\cdot;k)$ is Student's central $t$-distribution with $k$ degrees-of-freedom. The $P$-functionals, which do not require randomizers because of continuity, are given by
\begin{eqnarray*}
P_m(X_m,Y_m) & = & \inf\{\alpha \in (0,1): \delta_m(X_m,Y_m;\alpha) = 1\} 
 =  1 - \mathcal{T}(T_{mc};2(n_m-1)).
\end{eqnarray*}
The ROC function of $\Delta_m = \{\delta_m^*(X_m,Y_m;\alpha): \alpha \in [0,1]\}$, is given by, for $\alpha \in [0,1]$,
\begin{equation}
\label{ROC Function in two-sample example}
\rho_m(\alpha;\kappa,\sigma,n_m) = 1 - \mathcal{T}\left(t_{2(n_m-1);\alpha};2(n_m-1),\frac{\kappa}{\sigma\sqrt{2/n_m}}\right),
\end{equation}
where $\mathcal{T}(\cdot;k,\omega)$ is a $k$ degrees-of-freedom non-central $t$-distribution with non-centrality parameter $\omega$. The derivative of this ROC function, which is the density function of $P_m(X_m,Y_m)$ under $H_1$, is given by
\begin{equation}
\label{PDF under H1}
\rho_m^\prime(\alpha;\kappa,\sigma,n_m) = \frac
{\mathcal{T}^\prime\left(\mathcal{T}^{-1}(1-\alpha;2(n_m-1));2(n_m-1),\frac{\kappa}{\sigma\sqrt{2/n_m}}\right)}
{\mathcal{T}^\prime\left(\mathcal{T}^{-1}(1-\alpha;2(n_m-1));2(n_m-1)\right)},
\end{equation}
where $\mathcal{T}^\prime(\cdot;\cdot,\cdot)$ is the associated density function for Student's $t$-distribution.
These functions are easily evaluated using the {\tt R} \cite{R} objects {\tt pt}, {\tt dt}, and {\tt qt}.
Figure \ref{ROC and PDFs} plots the ROC curve and the pdfs of $P_m(X_m,Y_m)$ under $H_0$ and $H_1$ for the case with $n_m = 5$, $\kappa = 5$, $\sigma = 5$, and $\alpha=.05$. Observe that the ROC curve is above the 45-degree line, while the PDF of $P(X,Y)$ under $H_1$ is right-skewed. The behaviors of the ROC function $\rho(\cdot)$ and $\rho^\prime(\cdot)$ agree with those stated in Proposition \ref{prop-ROC properties}, though it should be pointed out that this two-sample $t$-test is not the uniformly most powerful test (UMP) of $H_0$ versus $H_1$, but it is the UMP unbiased decision function \cite{LehRom05,Sch95}.
\begin{figure}
\caption{The ROC function and PDFs under $H_0$ and $H_1$ of $P_m(X_m,Y_m)$ for the two-sample testing problem with $n_m = 5$, $\mu_0=0$, $\mu_1=5$, $\sigma = 5$, and $\alpha = .05$.}
\label{ROC and PDFs}
\includegraphics[width=\textwidth,height=2in]{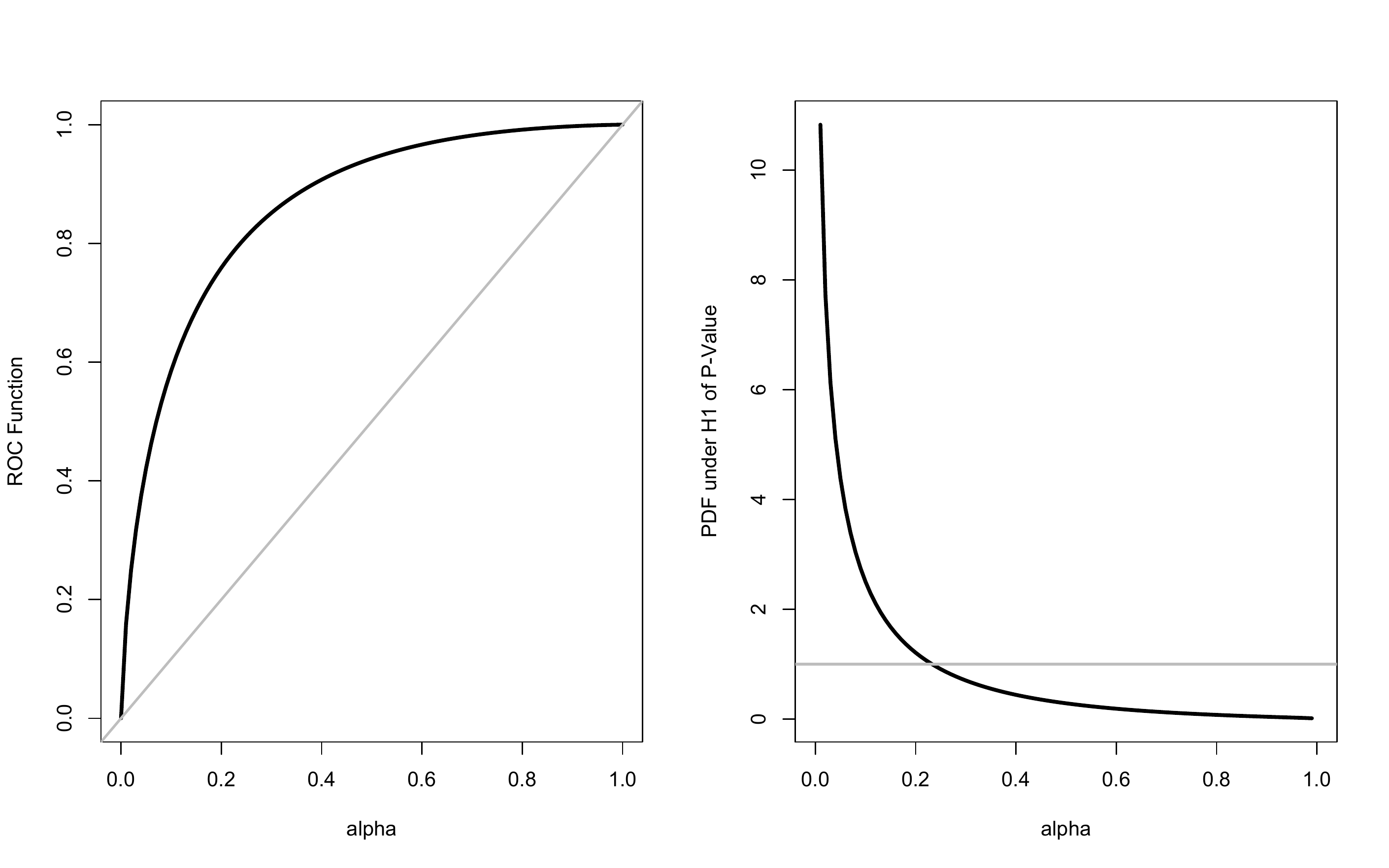}
\end{figure}

The plots in Figure \ref{under H0} represent the simulated results for the 100 scientists under $H_0$, while those in Figure \ref{under H1} are those under $H_1$. An LoS of $\alpha_0 = 0.05$, together with $\mu_0=0$, $\mu_1 = 2$, so $\kappa = 2$, and $\sigma = 5$, were utilized by each of these 100 scientists, though they could have also used varied LoSs. In each of these plot panels, the first represents the sequence of sample sizes $n_m$s; the second is the sequence of decisions  $d_m$s, together with the sequence of $p_m$s; the third is the histogram of the $p_m$s; and the fourth is the sequence of updated probabilities of $H_0$, which will be discussed in Section \ref{sec-Sequential Learning}. Looking at these plots, in particular the second panel in each figure,  assessing replicability of results, especially with just a few results, would be a difficult matter. Under $H_0$, decisions that coincide with not rejecting $H_0$ ($d_m=0$) would be assessed as replicable, but only if viewed under many replications; but, under $H_1$, the correct decisions of rejecting $H_0$ ($d_m=1$) may still be considered as not replicable since there are many more replications where $H_0$ is not rejected owing to the moderate powers of the decision functions used.

\begin{figure}
\caption{Sequences of the values of the decision function $\delta_m^*(\alpha)$ and values of the $P_m$-functional, under $H_0$, for the two-sample problem for 100 scientists. First panel is the sample size used for each sample, the second panel contains the decisions $d_m$s and the $p_m$s, the third panel is the histogram of the $p_m$s, and the fourth panel is the sequence of updated probabilities of $H_0$ based on updating on the $d_m$s and $p_m$s.}
\label{under H0}
\includegraphics[width=\textwidth,height=4in]{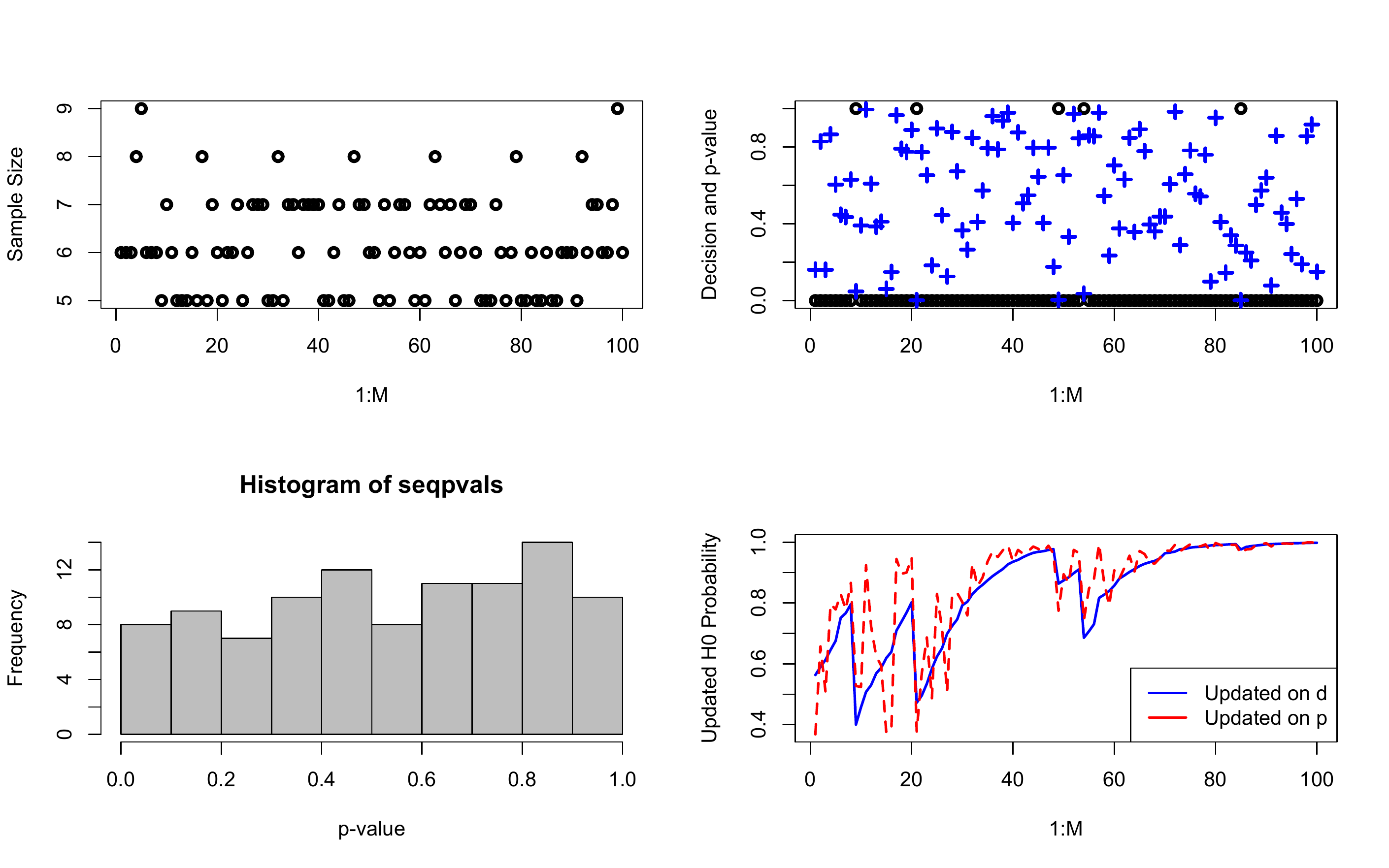}
\end{figure}

\begin{figure}
\caption{Sequences of the values of the decision function $\delta_m^*(\alpha)$ and values of the $P_m$-functional, under $H_1$, for the two-sample problem for 100 scientists. First panel is the sample size used for each sample, the second panel contains the decisions $d_m$s and the $p_m$s, the third panel is the histogram of the $p_m$s, and the fourth panel is the sequence of updated probabilities of $H_0$ based on updating on the $d_m$s and $p_m$s.}
\label{under H1}
\includegraphics[width=\textwidth,height=4in]{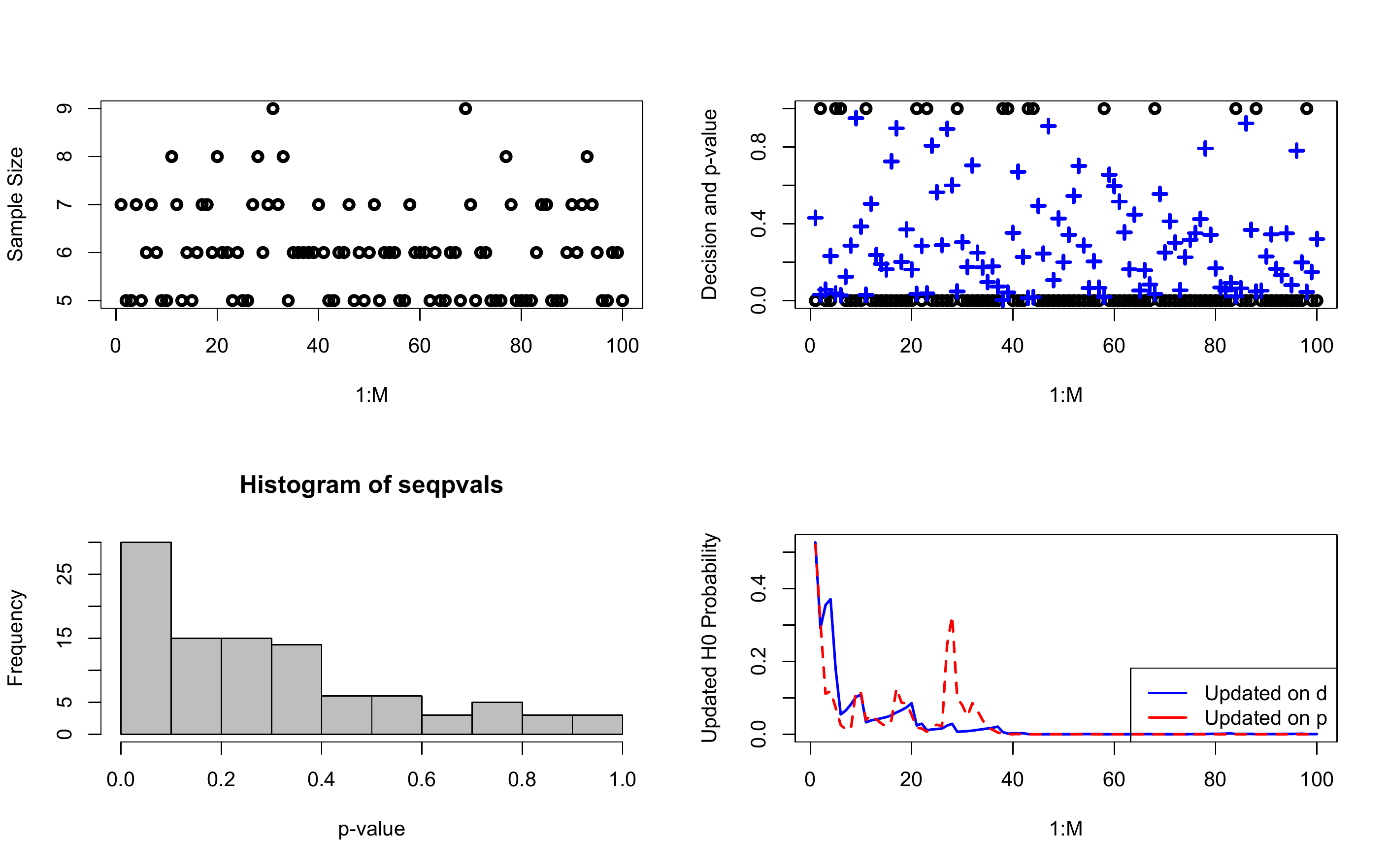}
\end{figure}

One could argue that since we are examining the results of 100 scientists, that we ought to have adjusted for the multiple testing. However, we wanted to mimic the realistic situation where each scientist performs their own study with or without knowledge of other scientists studies, hence chooses their own LoS, which will usually be 0.05 in the existing statistical decision-making paradigm.

\section{Knowledge Updating and Reporting Results}
\label{sec-Knowledge Updating}

So, how do we address the question of assessing the replicability of results? As the NAS report \cite{NASReport2019} alludes to, this should be done `in the context of an entire body of evidence, rather than an individual study or an individual replication'.  In a sense, each scientist's result should be utilized to update the current knowledge that we possess about $H_0$ and $H_1$. As mentioned in Section \ref{sect-introduction}, our knowledge of $H_0$ ($H_1$) could be represented by our subjective probability that $H_0$ ($H_1$) is true. As such upon seeing an individual scientist's result, be it the data $x$, the decision $d$, or the observed $p$, it should be used to update the current probability that $H_0$ is true. The updated probability that $H_0$ ($H_1$) is true will then represent the new  knowledge about $H_0$ ($H_1$).

Denote by $\kappa_0$ the prior probability that $H_0$ is true, and by $\kappa_1 = 1 - \kappa_0$ the prior probability that $H_1$ is true. If the observed data $x$ is reported, then by Bayes Theorem we obtain the posterior probabilities of $H_0$ and $H_1$ via:
\begin{eqnarray}
 \kappa_0(x) & = & \frac{\kappa_0 f_0(x)}{\kappa_0 f_0(x) + \kappa_1 f_1(x)} = \frac{1}{1 + (\kappa_1/\kappa_0) \Lambda(x)}; 
\label{update on x} \\
\kappa_1(x) & = & 1 - \kappa_0(x).  \nonumber
\end{eqnarray}

On the other hand, if only $\alpha_0$ and $d = \delta^*(x,u;\alpha_0)$ are reported, then the posterior probabilities are computed via Bayes theorem according to
\begin{eqnarray}
\kappa_0(d,\alpha_0) & = & \frac{\kappa_0 \alpha_0^d (1-\alpha_0)^{1-d}}
{\kappa_0 \alpha_0^d (1-\alpha_0)^{1-d} + \kappa_1 \rho(\alpha_0)^d [1 - \rho(\alpha_0)]^{1-d}} \label{update on d} \\
& = & \frac{1}{1 + (\kappa_1/\kappa_0) \Lambda_D(d;\alpha_0,\rho(\alpha_0))}; \nonumber \\
\kappa_1(d,\alpha_0) & = & 1 - \kappa_0(d,\alpha_0), \nonumber
\end{eqnarray}
where
\begin{equation}
\label{lik ratio based on D}
\Lambda_D(d;\alpha_0) = \left[\frac{\rho(\alpha_0)}{\alpha_0}\right]^d
\left[\frac{1-\rho(\alpha_0)}{1 - \alpha_0}\right]^{1-d}
\end{equation}
is the likelihood ratio based on observing $\delta^*(X,U;\alpha_0)$.
Observe that this updating will depend on the LoS used and the associated power or ROC value at the LoS of the decision rule. This updating rule also demonstrates the quantitative type of information obtained from a decision as a function of the LoS and power. For suppose that the decision is $d = 1$, i.e., reject $H_0$. Observe that as $\alpha_0 \rightarrow 0$, $\kappa_0(1,\alpha_0) \rightarrow \kappa_0/(\kappa_0 + \kappa_1 \rho^\prime(0+))$, and since $\rho^\prime(0+) > 1$ by Proposition \ref{prop-ROC properties}, then this limiting posterior probability is less than $\kappa_0$. That is, a ``reject $H_0$'' result with a small LoS provides evidence that $H_0$ is not true. In the other extreme, when $d = 1$ and $\alpha_0 \rightarrow 1$, $\kappa_0(1,\alpha_0) \rightarrow \kappa_0$, indicating that a ``reject $H_0$'' result, but with a large LoS, is {\em not} informative about whether $H_0$ is true or false. More could be said with respect to the interplay between the LoS $\alpha_0$ and the power $\rho(\alpha_0)$ and the information contained in the decision $d$. From the expression of $\Lambda_D$ in (\ref{lik ratio based on D}), observe that if $d = 1$ and $\rho(\alpha_0)/\alpha_0$ is much larger than 1, then this will contribute to a big change in the updated posterior probability that $H_0$ is true (will decrease this probability); whereas, if $1 < \rho(\alpha_0)/\alpha_0 \approx 1$, then $d = 1$ will not alter much this probability. When $\alpha_0$ is small, then $\rho(\alpha_0)/\alpha_0$ will tend to be large, while when $\alpha_0$ is large, then $\rho(\alpha_0)/\alpha_0$ will tend to be close to 1. 

The opposite holds true when the decision is $d = 0$, that is, ``do not reject $H_0$''. If $\alpha_0 \rightarrow 0$, then $\kappa_0(0,\alpha_0) \rightarrow \kappa_0$; while when $\alpha_0 \rightarrow 1$, then $\kappa_0(0,\alpha_0) \rightarrow \kappa_0/[\kappa_0 + \kappa_1 \rho^\prime(1-)]$, and since $\rho^\prime(1-) < 1$ by Proposition \ref{prop-ROC properties}, this limit value exceeds $\kappa_0$. That is, a ``do not reject $H_0$'' decision with small $\alpha_0$ is not informative about $H_0$; whereas, it is informative about $H_0$ when $\alpha_0$ is large. Again, more could be said from $\Lambda_D$ since when $d = 0$ the relevant factor is $[1 - \rho(\alpha_0)]/[1-\alpha_0]$. When $\alpha_0$ is small, then this ratio is close to 1, so that it will not alter much the prior probability of $H_0$ in the updating; whereas, if $\alpha_0$ is large and power $\rho(\alpha_0)$ at $\alpha_0$ is large, then the ratio becomes small, thus contributing to a big change from prior to posterior probabilities.

These considerations indicate that it is imperative to accompany the decision $d$, if a decision is actually needed to be made, with the value of
$\Lambda_D(d;\alpha_0)$
or, equivalently, by $\log \Lambda_D(d;\alpha_0)$, a deviance associated with $\delta^*(\alpha_0)$, which measures the quality of the information about $H_0$ and $H_1$ provided by the realization of the decision function. This shows that specific thresholds on the LoS, such as the 0.05 threshold, are really not needed since whatever LoS $\alpha_0$ value a scientist uses, provided it is not data-determined, will be automatically factored into the summary measure $\Lambda_D(d;\alpha_0)$ or $\log \Lambda_D(d;\alpha_0)$. Note that the power $\rho(\alpha_0)$ of the decision function at the chosen $\alpha_0$ is a critical component in this summary measure. In a similar vein that an estimate of the standard error of an estimator accompanies the estimate arising from the estimator to serve as a measure of the precision of the estimate, then $\log \Lambda_D(d;\alpha_0)$ could be viewed as a measure of the quality of the realized decision, with a large value of $|\log \Lambda_D(d;\alpha_0)|$ indicating the decision $d$ is of high quality.

If one wants to use the value $p$ of the $P$-functional to update the current knowledge of $H_0$ and $H_1$, then their posterior probabilities are computed via Bayes theorem using
\begin{eqnarray}
\kappa_0(p) & = & \frac{\kappa_0}
{\kappa_0 + \kappa_1 \rho^\prime(p)} = \frac{1}{1 + (\kappa_1/\kappa_0) \rho^\prime(p)}; \label{update on p} \\
\kappa_1(p) & = & 1 - \kappa_0(p).
\end{eqnarray}
Let us examine the information provided by different values of $p$. Recall that $\rho^\prime(\alpha)$ is decreasing in $\alpha$ under $H_1$ with $\rho^\prime(0+) > 1$ and $\rho^\prime(1-) < 0$ by Proposition \ref{prop-ROC properties}, so that there exists an $\alpha^*$ such that $\rho^\prime(\alpha^*) = 1$. Thus, for $p_1 < p_2 < \alpha^*$, we will have that
$\kappa_0(p_1) < \kappa_0(p_2) < \kappa_0,$
indicating that smaller observed $P$-values are more indicative that the $H_0$ is not true. When $\alpha^* < p_1 < p_2$, then
$\kappa_0 < \kappa_0(p_1) < \kappa_0(p_2),$
so in this case, larger values of the $P$-functional are more indicative that $H_0$ is true. These results mathematically depict the information contained in the $P$-functional about $H_0$ and $H_1$, and coincides with our intuition that {\em small $p$-values point more towards $H_1$ being true, whereas large $p$-values are indicative of $H_0$ being true.}  However, note that this conclusion will depend on the $\alpha^*$ satisfying $\rho^\prime(\alpha^*) = 1$, and such an $\alpha^*$ will be dependent on the ROC function of the decision process, which in turn will depend on the effect size associated with the $H_0$ and $H_1$ and usually on the sample size used in the study. 

Based on these considerations, just reporting the value $p$ of the $P$-functional, which is the {\em modus operandi} when used in Fisher's NHST approach and in many statistical hypothesis-testing procedures, apparently is not a good approach and could in fact mislead, since users could make conclusions about the strength of evidence for or against $H_0$ simply based on the magnitude of $p$, e.g., when $p = 0.0001$ that $H_0$ is highly implausible. This fallacy is demonstrated by considering two simple alternative hypotheses, which are in the same direction. For the same observed data $(x,u)$, the $P$-functional value will not change under both alternative hypotheses since it is computed only under $H_0$. But, clearly, the information content about $H_0$ in the realized $P$-functional differs under the two alternative hypotheses. More ominously, one would then think that a small $p$ will indicate more support for the more extreme alternative hypothesis, but this turns out to be not the case in general. Thus, there is something missing when we only report the realized value $p$, as is done in NHST and in many statistical hypothesis testing situations.  

So, what might be a better approach? We propose to report as summary measure $\rho^\prime(p)$ or $\log \rho^\prime(p)$, with a small [large] value indicating more [less] support towards $H_0$. We point out that this is different from Greenland's \cite{Gre19} proposal of reporting the $S$-value, the negative of the logarithm of the $p$-value itself, which will still not reflect anything about the alternative hypothesis. An advocate of the NHST approach might argue that this cannot be implemented since there is no alternative hypothesis under the NHST approach since it only asks if the data is consistent with the null model. But, without an alternative model the question of when an observation is inconsistent with the null model becomes problematic. For instance, suppose that the null model is that $X$ has a standard uniform distribution. One might conclude that an outcome $x$ in $[0,.05]$ is an extreme realization under the model, but what about an outcome in $x$ in $[.475,.525]$?  Both events have the same probability of $.05$ under the null model! The Neyman-Pearson framework, which could be viewed as an alteration of the NHST framework, improves on the NHST approach since it insists on taking into account an alternative model, thereby eliminating the indeterminacy of what will be considered as more extreme data realizations under the null model in light of the alternative model.
In addition, we point out that it is possible, for example, to have $p = .0001$ -- which would ordinarily be interpreted as strongly supporting $H_1$ -- but at the same time to have $\rho^\prime(.0001) << 1$, which is indicative of more support for $H_0$. The summary measure $\rho^\prime(p)$, or equivalently $\log \rho^\prime(p)$, is the $P$-based likelihood ratio, so it accounts for the distributions of $P$ under {\em both} $H_0$ and $H_1$; whereas, the current practice of reporting $p$ in isolation, provides a rudderless summary measure since it does {\em not} have the proper context to assess its informativeness towards $H_0$ or $H_1$. Any value of $p$ with $\rho^\prime(p) = 1$, or, equivalently $\log \rho^\prime(p) = 0$, is totally uninformative about $H_0$ and $H_1$, and such a $p$ could possibly be close to zero, close to one, or somewhere in the middle portion of $[0,1]$. Now, if one wants to make a decision between $H_0$ and $H_1$ using the value $p$, then he must specify an LoS $\alpha_0$, determined independently of the observed data, and decide $d = 0 [1]$ if $p  > [\le] \alpha_0$, but he must then accompany this with $\Lambda_D(d;\alpha_0)$ or $\log \Lambda_D(d;\alpha_0)$ as indicated earlier regarding the reporting of the result of a decision function. It is also worth noting that the value of $\rho^\prime(p)$ plays a major role in the optimal choice of LoSs to use in a generalized Benjamini-Hochberg \cite{BenHoc95} false discovery rate (FDR) controlling procedure to optimize the global power (see Theorem 4.3 in \cite{PenHabWu11}) in a multiple testing setting, thus lending credence that the value of $\rho^\prime(p)$ is the more important quantity instead of just the value of $p$ itself. We note in passing that many multiple-testing procedures, including the BH procedure \cite{BenHoc95}, that corrects for multiplicity are just based on the values of the $P$-functional for each of the multiple tests, thereby creating a possible opening for improving such procedures by considering instead the values of $\rho^\prime(p)$ for each of the multiple tests; see also \cite{PenHabWu15}.

In order to implement the updating rules based on observing $\delta(x,u;\alpha_0) = d$ and $P(x,u) = p$, the quantities $\rho(\alpha_0)$ and $\rho^\prime(p)$ need to be known. In principle, knowledge of $f_0$ and $f_1$, which are needed for the updating based on observing $X = x$, is sufficient to determine $\rho(\alpha)$ and hence $\rho^\prime(\alpha)$. Since our main focus in this paper is the simple $H_0$ versus the simple $H_1$ setting, these functions will be completely known.

However, let us briefly consider a situation with a composite $H_1$. Thus, suppose that $X$ has pdf $f$ with respect to a dominating measure $\nu$ and which belongs to a family of pdfs $\mathfrak{F} = \{f(x;\theta): \theta \in \Theta \subset \Re\}$ satisfying a monotone likelihood ratio (MLR) property in the one-dimensional statistic $S(x)$ (see, for instance, \cite{LehRom05} for discussions of the MLR property), and of interest is to decide between the simple null hypothesis $H_0: \theta = \theta_0$ and the composite alternative hypothesis $H_1: \theta > \theta_0$. There will be a decision process $\Delta = \{\delta(x,u;\alpha): \alpha \in (0,1)\}$ for $H_0$ versus $H_1$ of form
\begin{displaymath}
\delta(x,u;\alpha) = I\{S(x) > c(\alpha)\} + I\{S(x) = c(\alpha); u \le \gamma(\alpha)\},
\end{displaymath}
where $c(\alpha)$ and $\gamma(\alpha)$ only depends on the distribution of $S(X)$ under $H_0$. For this decision process there will be an associated ROC function which will depend on both $\alpha$ and $\theta_1$, where $\theta_1$ is the true value of $\theta$:
$\rho(\theta_0,\theta_1) = \{\rho(\alpha;\theta_0,\theta_1): \alpha \in (0,1)\}$, with
$$\rho(\alpha;\theta_0,\theta_1) = \langle \delta(\alpha), \Lambda(\theta_0,\theta_1) \rangle$$
where $$\Lambda(\theta_0,\theta_1) = \Lambda(x;\theta_0,\theta_1) = f(x;\theta_1)/f(x;\theta_0) = g[S(x); \theta_0,\theta_1]$$ for some $g(s;\theta_0,\theta_1)$ which is non-decreasing in $s$ whenever $\theta_1 > \theta_0$, and with the inner product defined via 
$$\langle g_1, g_2 \rangle = \int_0^1 \int_\mathfrak{X} g_1(x,u)g_2(x,u) f(x;\theta_0) \nu(dx) du.$$
Given the realized decision $d$ based on decision function $\delta(\cdot,\cdot;\alpha)$ and realized $p$ of the $P$-functional, we let
\begin{eqnarray*}
l_D(\theta_1;\theta_0,d,\alpha) & = &
d \log\left[\frac{\rho(\alpha;\theta_0,\theta_1)}{\alpha}\right] + (1 - d) \log\left[\frac{1-\rho(\alpha;\theta_0,\theta_1)}{1-\alpha}\right]; \\
l_P(\theta_1;\theta_0,p) & = & \log \rho^\prime(p;\theta_0,\theta_1),
\end{eqnarray*}
denote the log-likelihood {\em ratios} based on $D=d$ and $P=p$, respectively.
The dependence of each of these functions on $(\theta_0,\theta_1)$ will usually be through a one-dimensional parametric function $\xi(\theta_0,\theta_1)$, which represents the `effect size' or the `distance' between $\theta_0$ and $\theta_1$. As summary plots to accompany the decision $d$ or the realized $P$-functional $p$, we could provide plots of $(\xi(\theta_0,\theta_1),l_D(\theta_1,;\theta_0,\alpha,d))$ or $(\xi(\theta_0,\theta_1),l_P(\theta_1;\theta_0,p))$ as $\theta_1$ varies from $\theta_0$. Such plots could provide information about the quality of the information about $H_0: \theta = \theta_0$ versus the possible values under $H_1$.  Values of these functions that are close to zero will not be informative, whereas large values will indicate support of $H_1$, while small values will indicate support for $H_0$.

We illustrate the above ideas using the two-sample setting utilized in the simulation study in Section \ref{sec-Replicability}. For a given $\mu_0$, $\mu_1$, and $\sigma$, the {\em standardized effect size} is defined as $\xi = (\mu_1 - \mu_0)/\sigma$. Given an $n$ and $\alpha$, together with the value of $\xi$, we obtain an expression for $l_D(\xi;d,\alpha)$ and also $l_P(\xi;p)$. We could plot in 3-dimensional space the mappings $(\xi,\alpha) \mapsto l_D(\xi;d,\alpha)$ and $(\xi,p) \mapsto l_P(\xi;p)$, or we could also just create contour plots of these mappings. For aesthetic purposes, it is better to plot with respect to $(\xi,\log(\alpha/(1-\alpha))$ the $l_D(\xi;d,\alpha)$ and  $(\xi,\log(p/(1-p))$ the $l_P(\xi;p)$. Figure \ref{fig-contour plots} provides these contour plots for $n=10$ and $n=20$.
%
%
The contour plots associated with $l_D(\xi,\alpha;d)$ provide information regarding the quality of information about $H_0$ and $H_1$ that could be obtained from observing either a decision of $d = 0$ (do not reject $H_0$) or $d = 1$ (reject $H_0$) for pairs of values of $(\xi,\log(\alpha/(1-\alpha))$; whereas, the contour plot associated with $l_P(\xi,p)$ provides information about the quality of information regarding $H_0$ and $H_1$ that one obtains for $(\xi,\log(p/(1-p))$. Just for reference, note that $\log(.05/(1-.05)) = -2.9444$.

\begin{figure}
\caption{Contour plots of $l_D(\xi;d,\alpha)$ and $l_P(\xi;p)$ with respect to the effect size $\xi$ and $\log(\alpha/(1-\alpha))$ or $\log(p/(1-p))$ for $n = 10$ (first column) and $n = 20$ (second column) in the two-sample problem.}
\label{fig-contour plots}
\begin{center}
\begin{tabular}{|c|c|} \hline
Sample Size $n = 10$ & Sample Size $n = 20$ \\ \hline
\includegraphics[width=2.75in,height=2.1in]{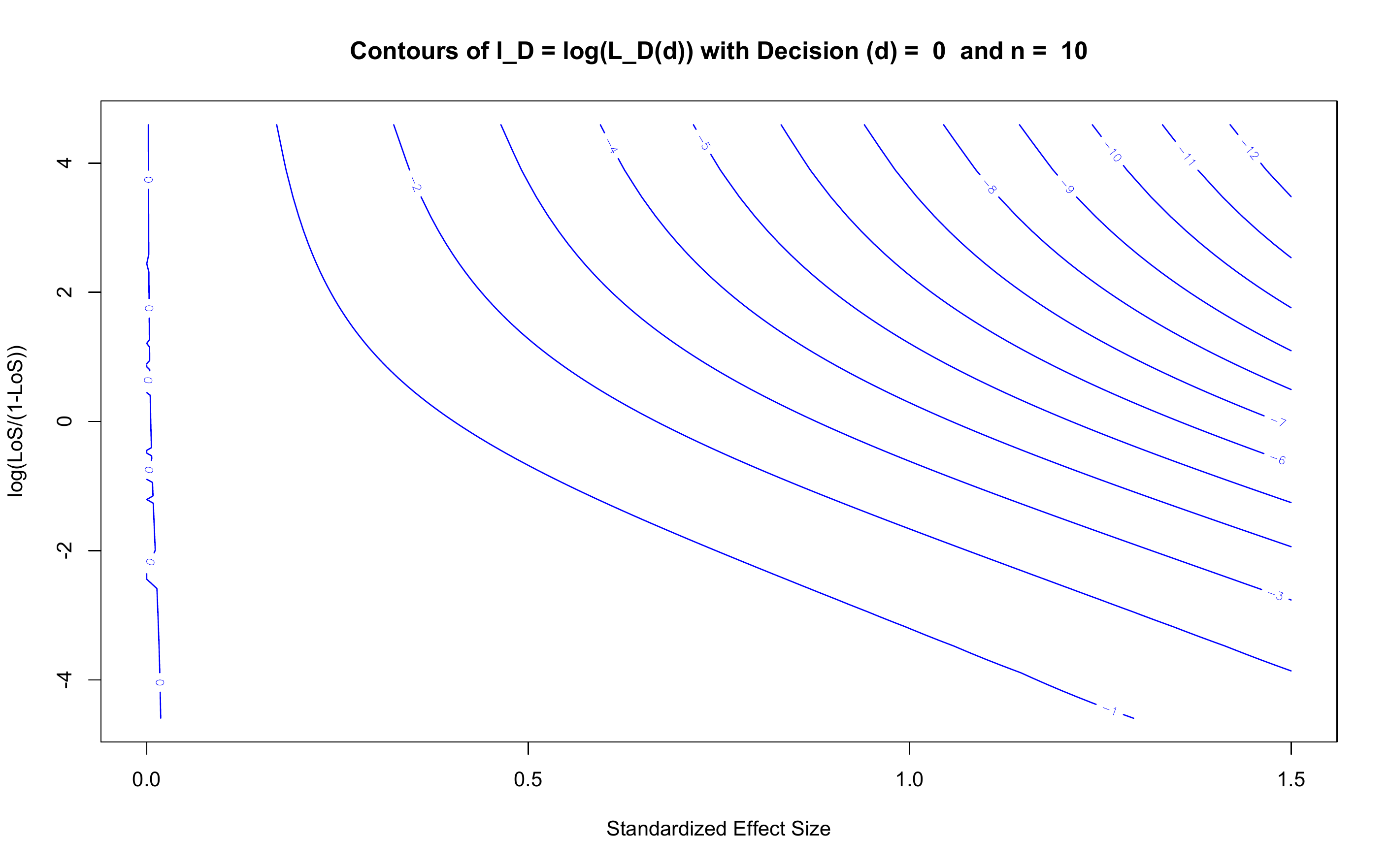} & 
\includegraphics[width=2.75in,height=2.1in]{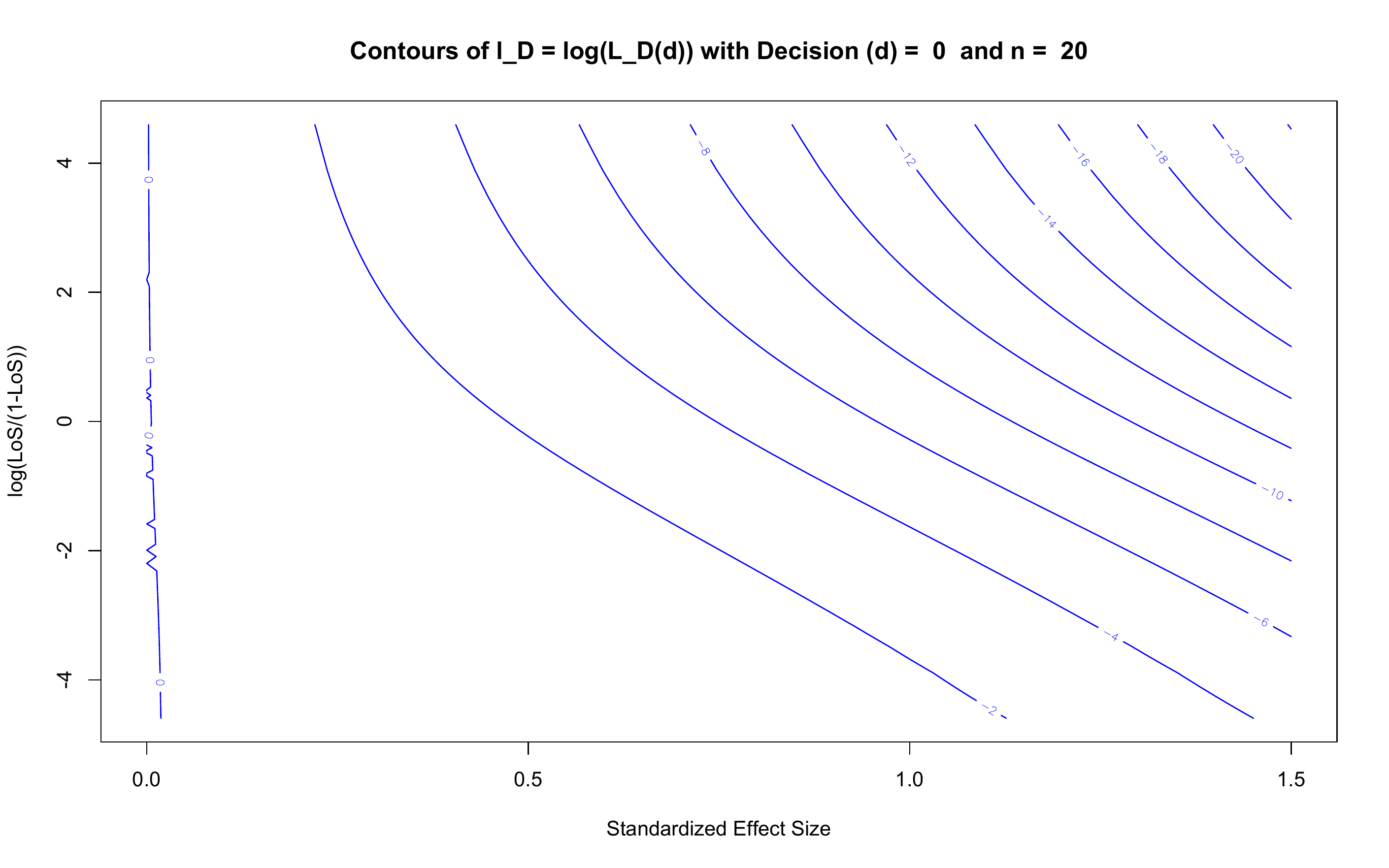} \\
\includegraphics[width=2.75in,height=2.1in]{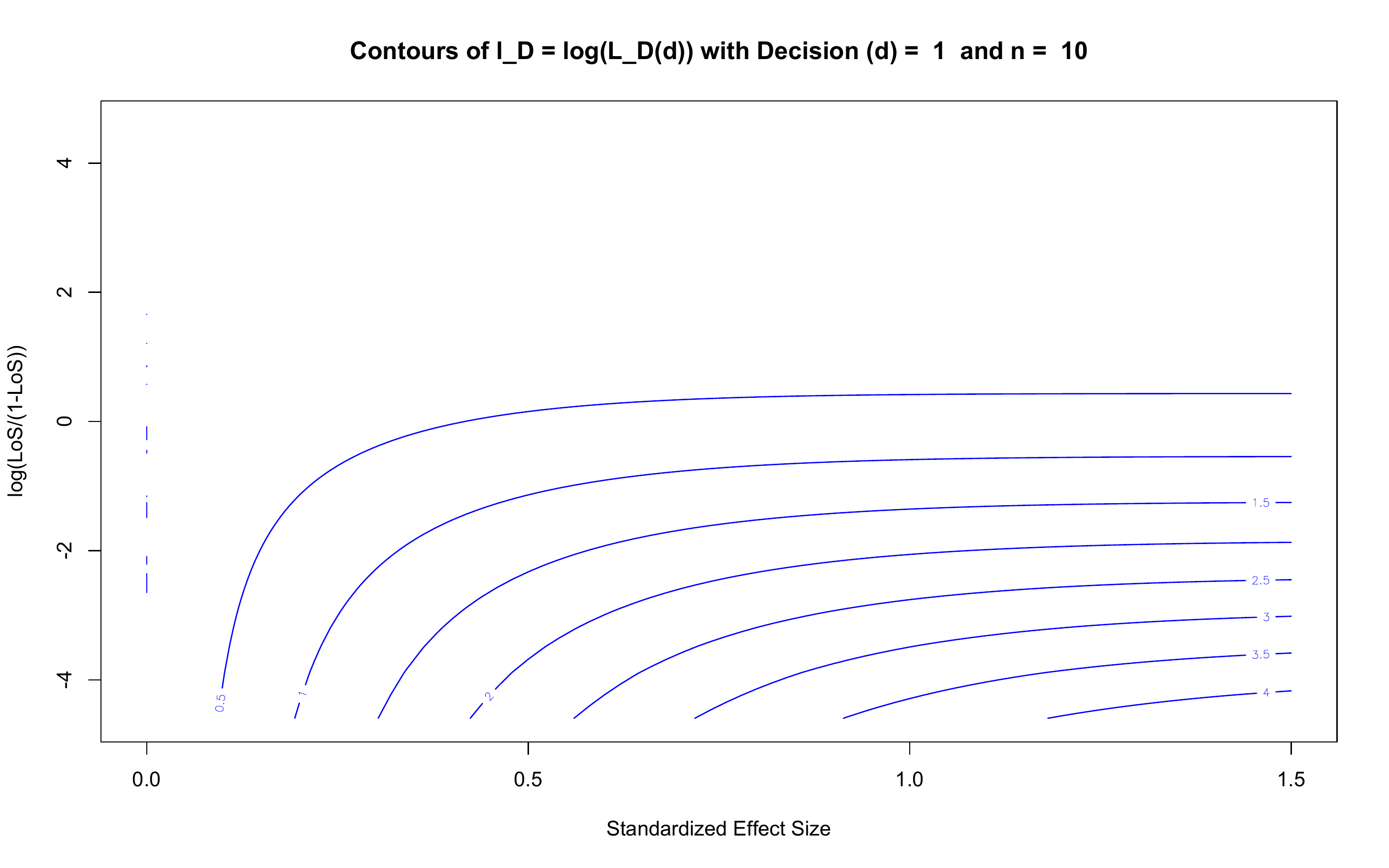} &
\includegraphics[width=2.75in,height=2.1in]{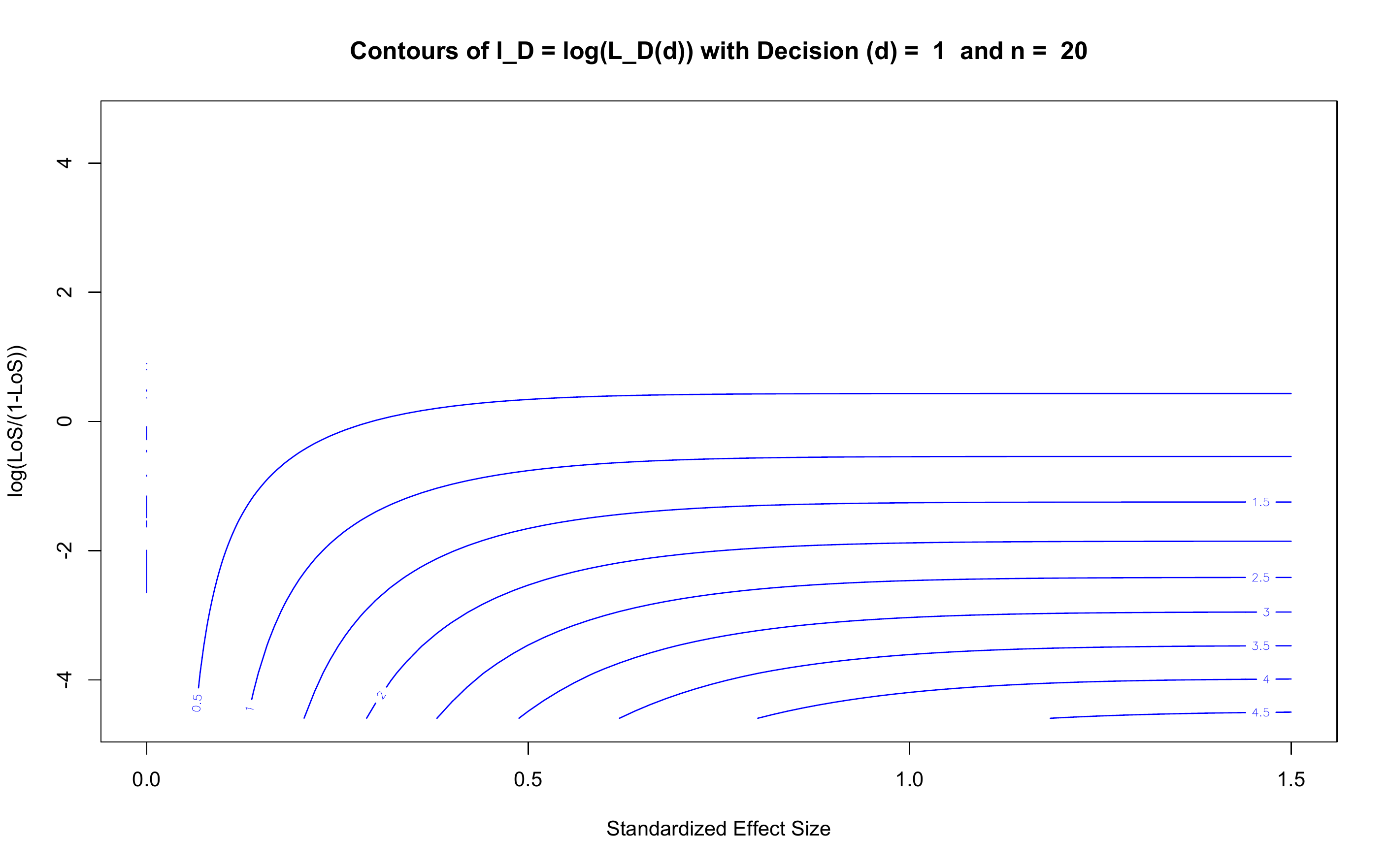} \\
\includegraphics[width=2.75in,height=2.1in]{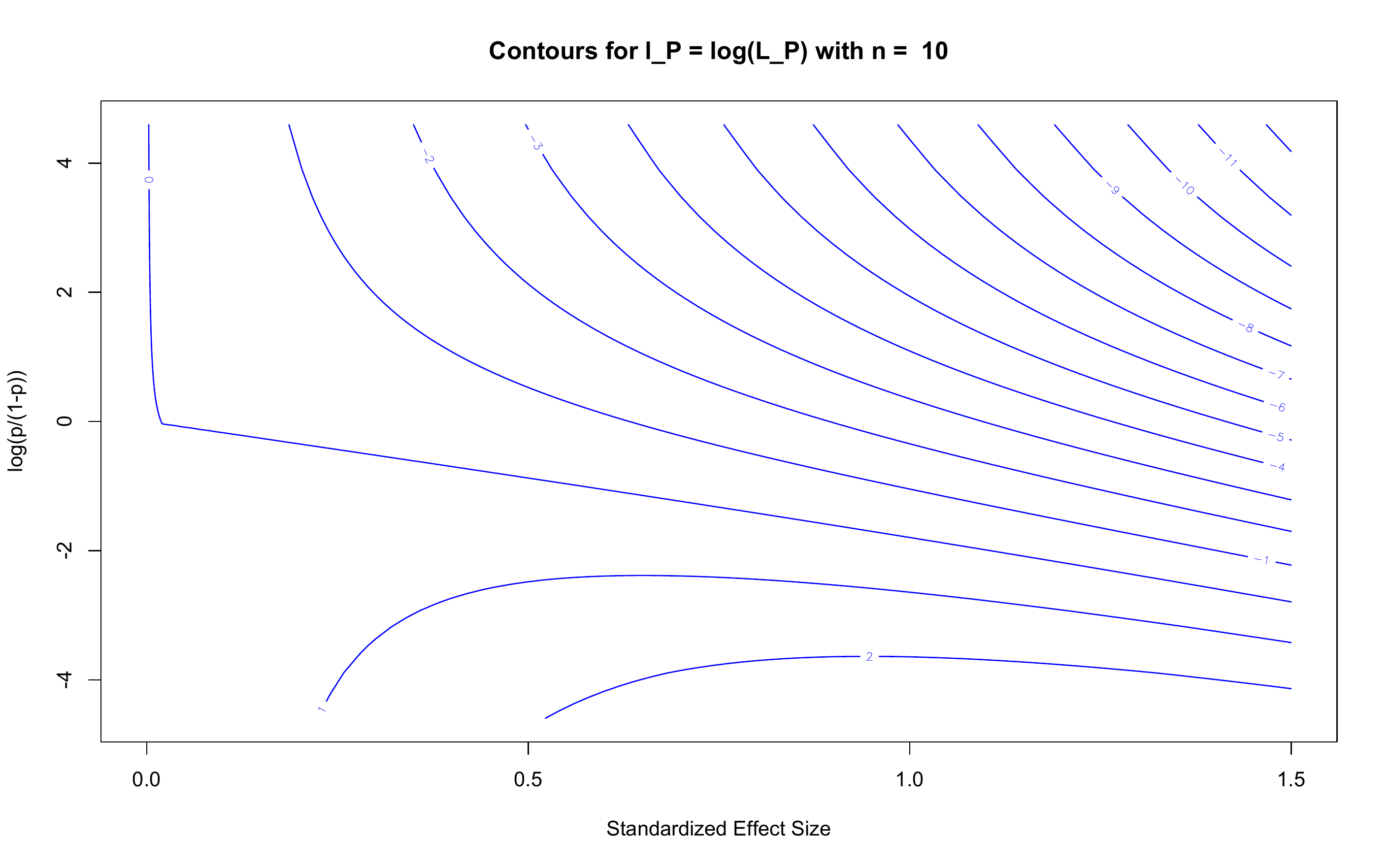} &
\includegraphics[width=2.75in,height=2.1in]{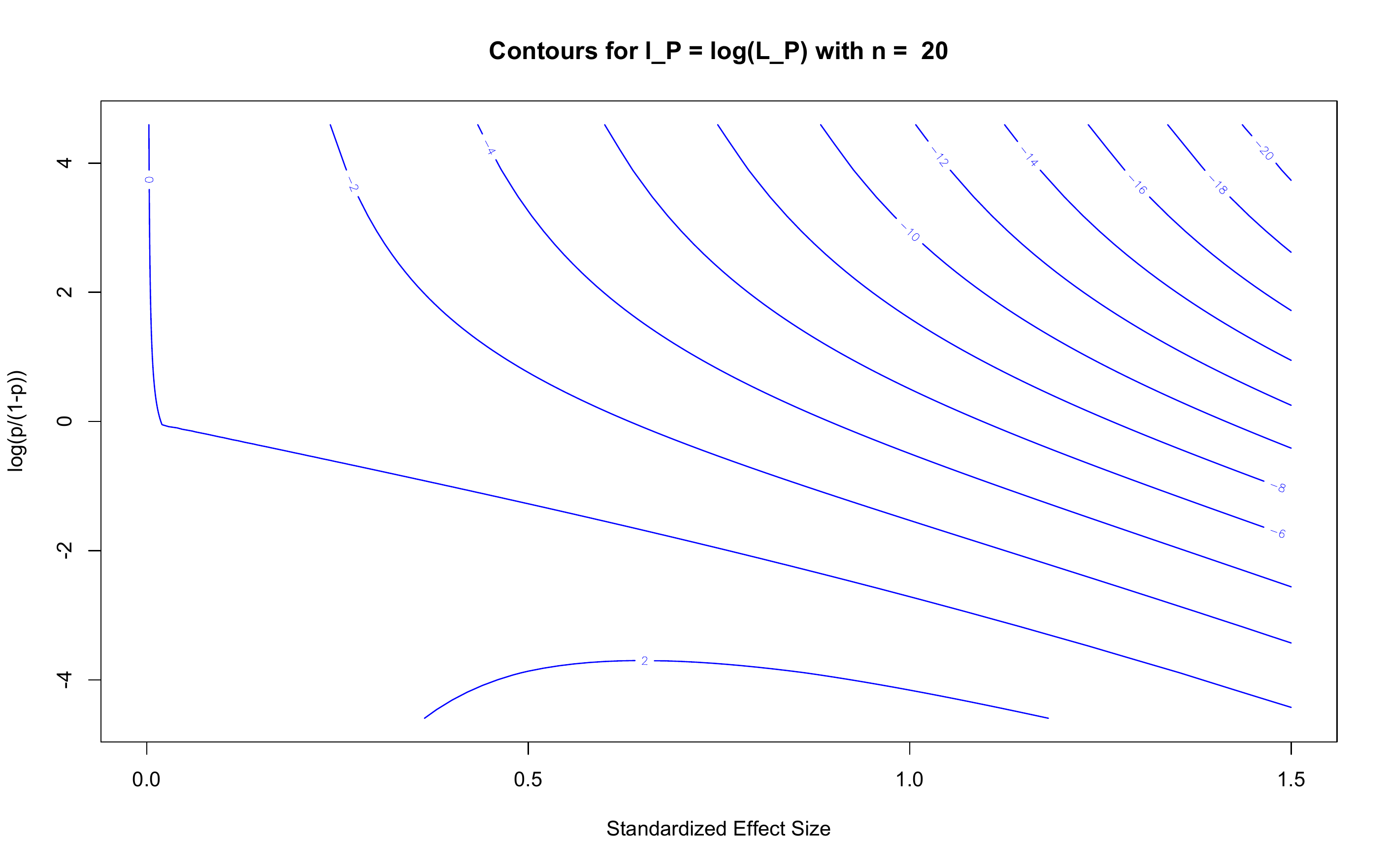} \\ \hline
\end{tabular}
\end{center}
\end{figure}

Next, to illustrate what could transpire in practice, we generated two data sets each under the null hypothesis model ($H_0: \mu_0 = 0, \mu_1 = 0$) and under the specific alternative hypothesis model ($H_1: \mu_0 = 0, \mu_1 = 5$) with $\sigma = 5$ with $n = 10$ and $n = 20$. Based on the generated two-sample data sets, we determined the realized decision, $d$, under an LoS of $\alpha = .05$, and the value of $p$. We provide the comparative boxplots of the two samples, together with the plot of $\xi \mapsto l_D(\xi;d,\alpha)$ and $\xi \mapsto l_P(\xi;p)$ for different values of $\xi$. The realized $d$ and $p$ are indicated at the top of the second and third plot panels in Figures \ref{fig-all realizations n = 10} and \ref{fig-all realizations n = 20}.

\begin{figure}
\caption{Sample realizations (two each) from the two-sample model under $H_0$ and $H_1$ together with the plots of the log-likelihood ratio summaries as a function of effect size with sample size of $n = 10$.}
\label{fig-all realizations n = 10}
\begin{tabular}{c} \hline\hline
Data Generating Model Under $H_0$: $\mu_0 =0, \mu_1 = 0, n=10, \sigma = 5$ \\ \hline
Sample Realization \#1 \\
\includegraphics[width=\textwidth,height=1.25in]{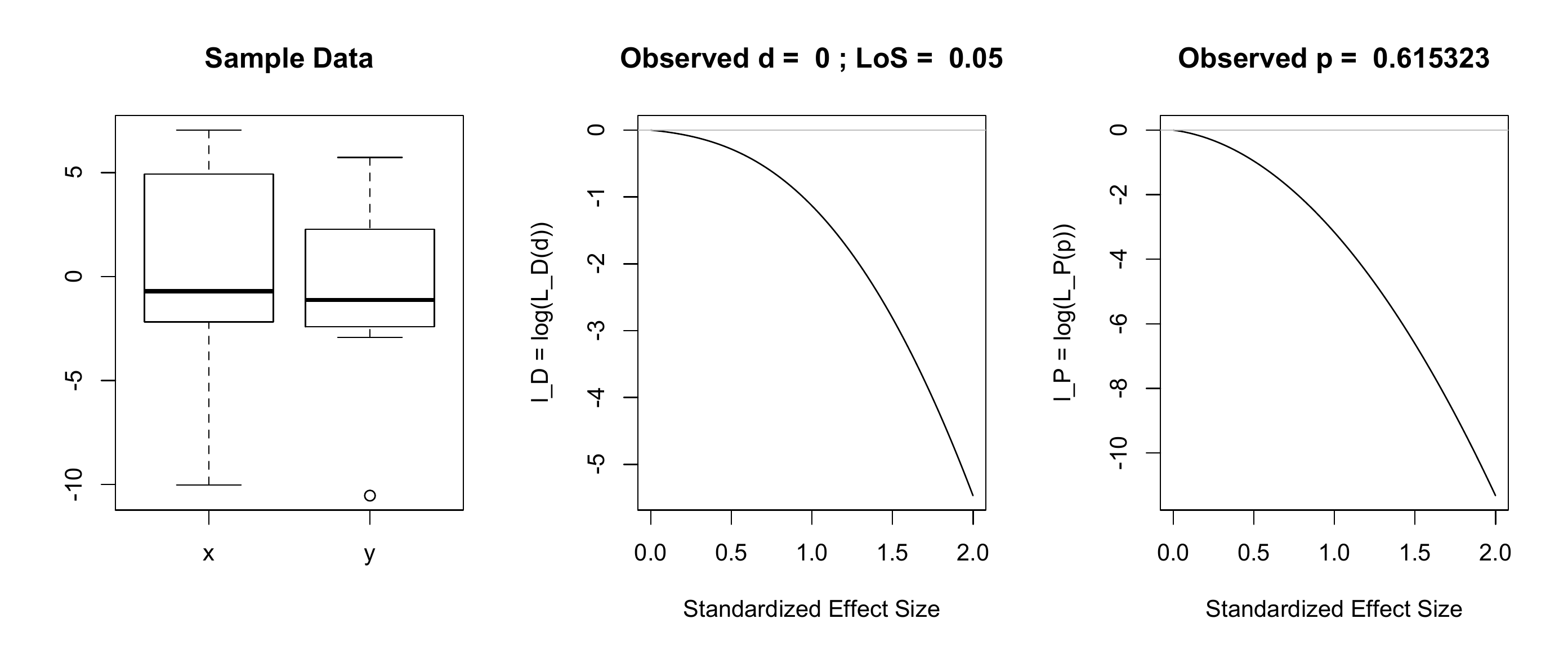} \\ \hline \\
Sample Realization \#2 \\
\includegraphics[width=\textwidth,height=1.25in]{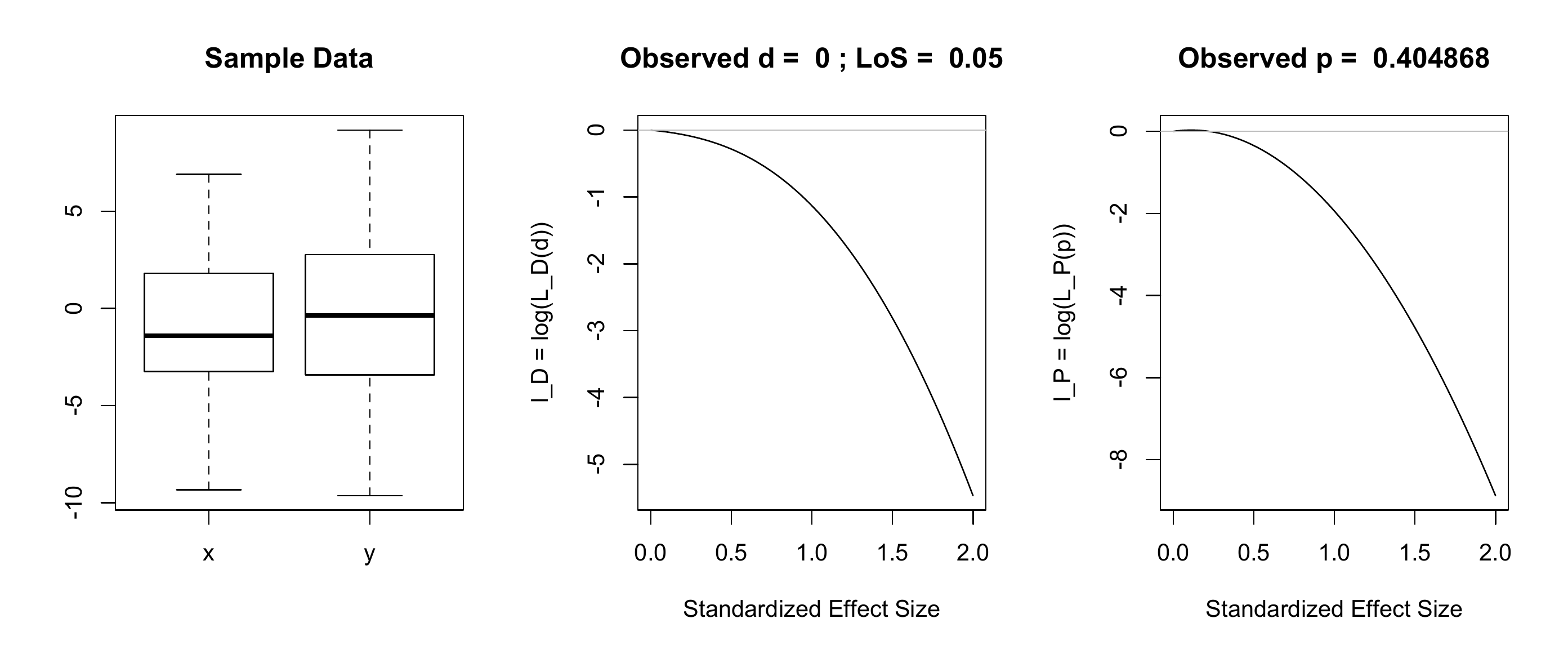}  \\ \hline\hline
Data Generating Model Under $H_1$: $\mu_0 =0, \mu_1 = 5, n=10, \sigma = 5$ \\ \hline
Sample Realization \#1 \\
\includegraphics[width=\textwidth,height=1.25in]{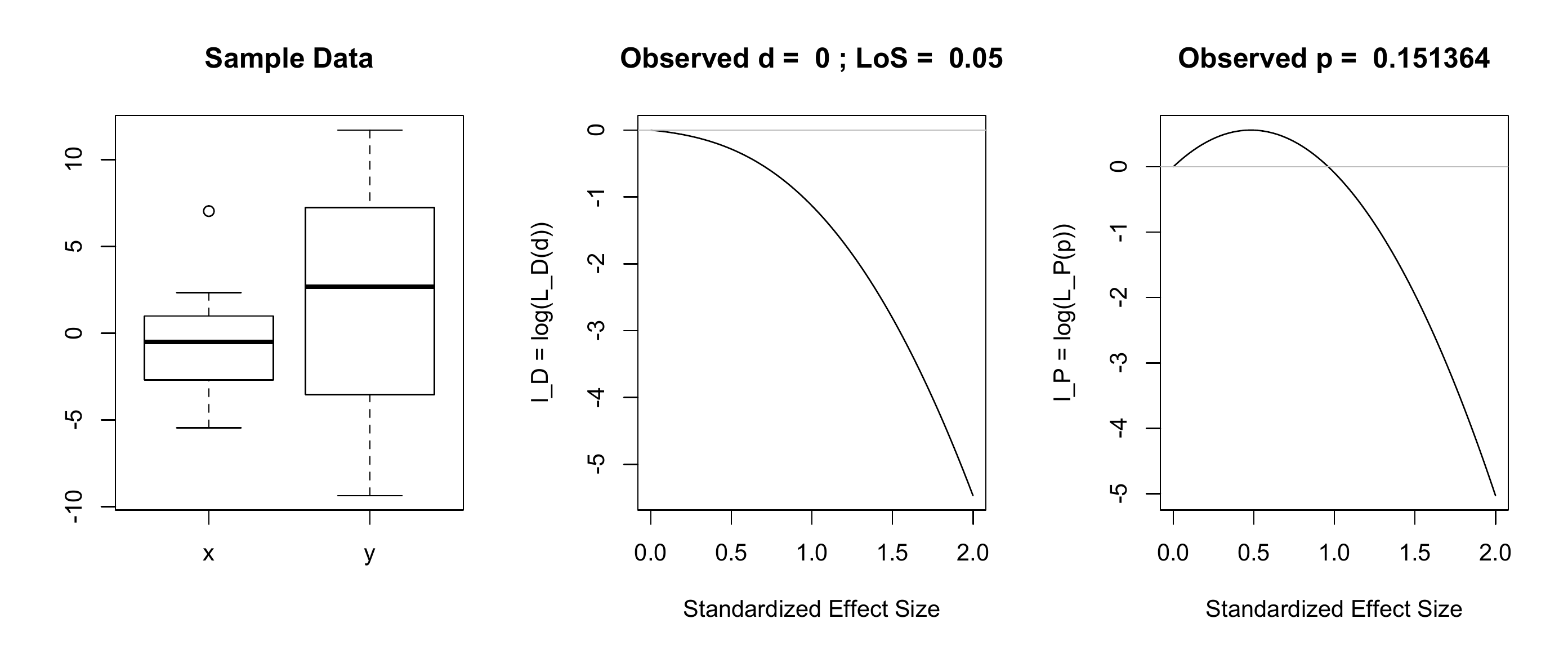} \\ \hline \\
Sample Realization \#2 \\
\includegraphics[width=\textwidth,height=1.25in]{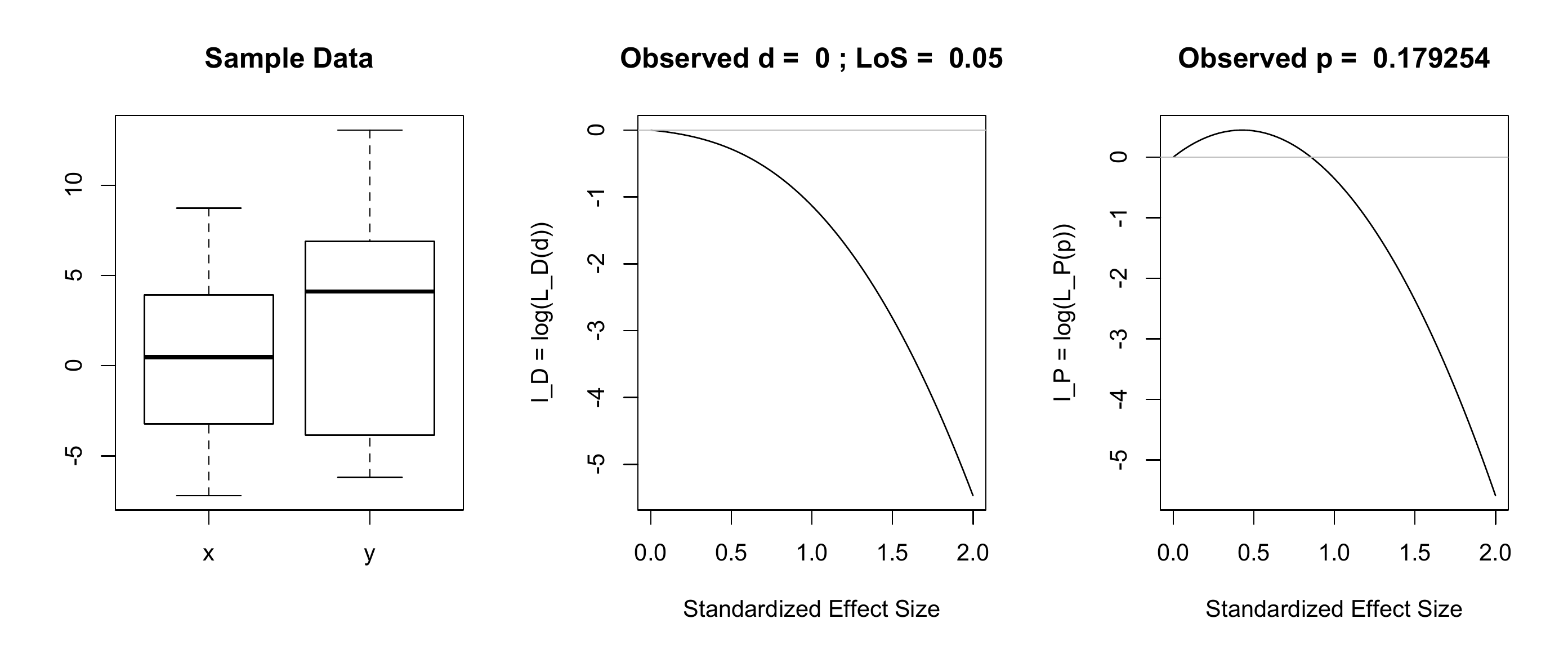}  \\ \hline\hline
\end{tabular}
\end{figure}

\begin{figure}
\caption{Sample realizations (two each) from the two-sample model under $H_0$ and $H_1$ together with the plots of the log-likelihood ratio summaries as a function of effect size with sample size of $n = 20$.}
\label{fig-all realizations n = 20}
\begin{tabular}{c} \hline\hline
Data Generating Model Under $H_0$: $\mu_0 =0, \mu_1 = 0, n=20, \sigma = 5$ \\  \hline
Sample Realization \#1 \\
\includegraphics[width=\textwidth,height=1.25in]{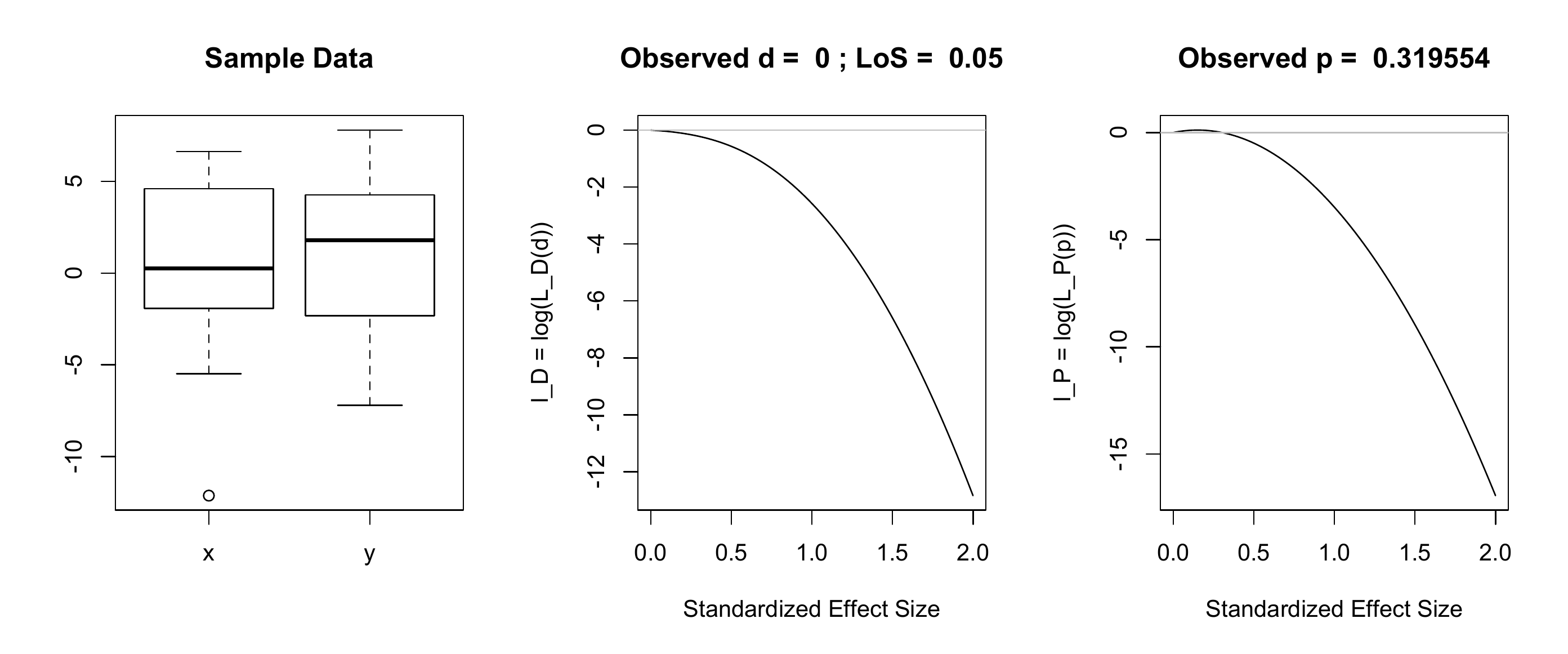} \\ \hline \\
Sample Realization \#2 \\
\includegraphics[width=\textwidth,height=1.25in]{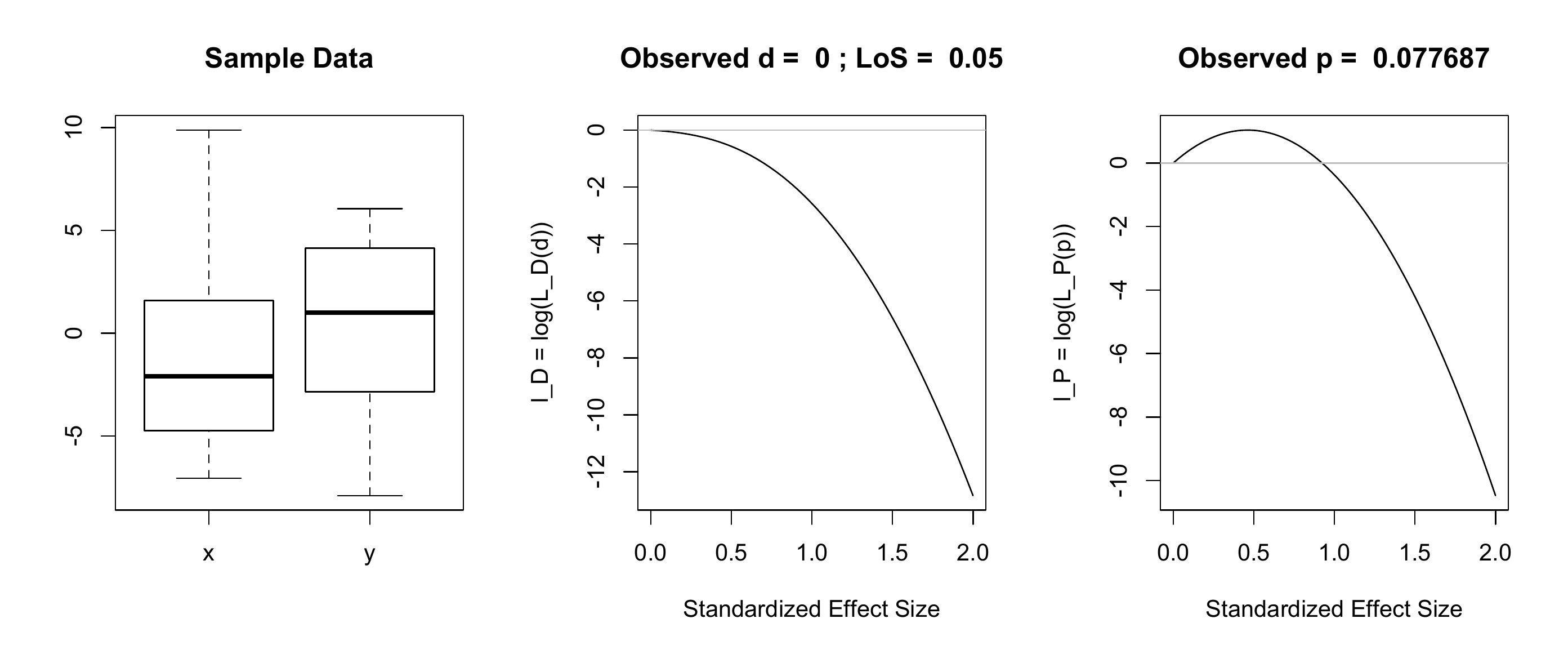}  \\ \hline\hline
Data Generating Model Under $H_1$: $\mu_0 =0, \mu_1 = 5, n=20, \sigma = 5$ \\ \hline
Sample Realization \#1 \\
\includegraphics[width=\textwidth,height=1.25in]{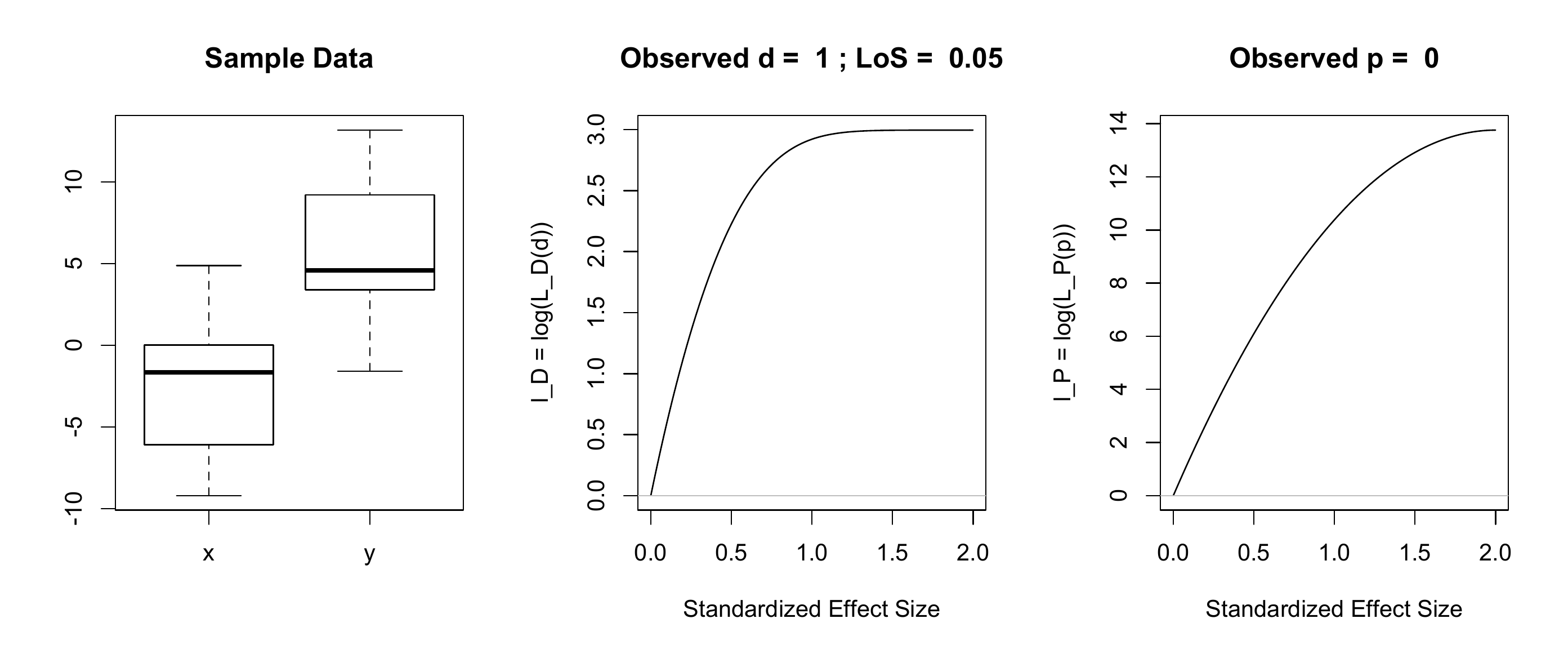} \\ \hline \\
Sample Realization \#2 \\
\includegraphics[width=\textwidth,height=1.25in]{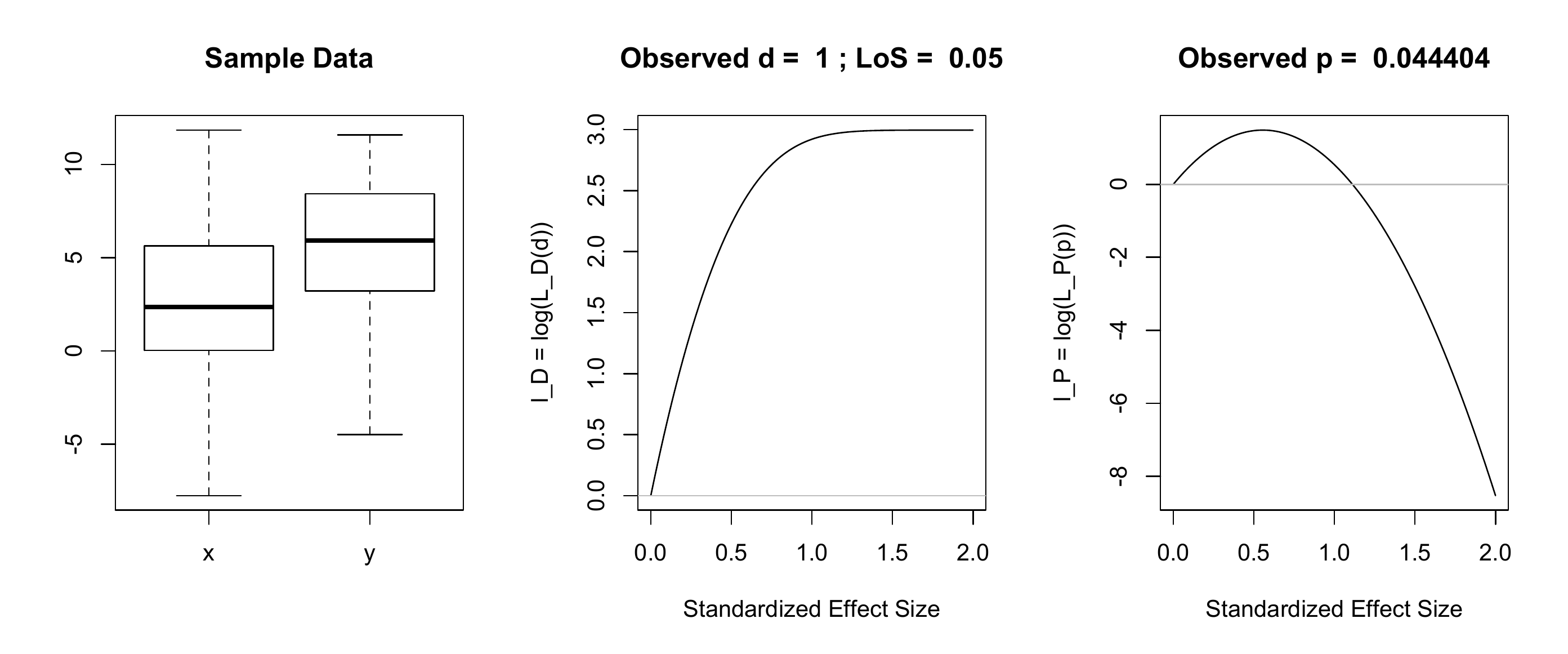}  \\ \hline\hline
\end{tabular}
\end{figure}

In Figure \ref{fig-all realizations n = 10}, both sample realizations with $n=10$, which were generated under the null model, led to $d = 0$ at $\alpha = .05$ and the realized $P$ were $.6153$ and $.4048$. Looking at the plots of $\xi \mapsto l_D(\xi;d,\alpha)$ and $\xi \mapsto l_P(\xi;p)$, these all point to support for $H_0$, with the strength of support increasing as $\xi$ increases. Also in Figure \ref{fig-all realizations n = 10} where the $y$-sample was generated from a normal with mean 5 still with $n=10$, so under $H_1$, the two sample realizations still yielded $d = 0$ at $\alpha = .05$, which are both Type II errors. The plots of $\xi \mapsto l_D(\xi;d=0,\alpha)$ were below zero, indicating support for $H_0$. Note by the way that the dependence on the observed data of $l_D(\xi;d,\alpha)$ is only through the value of $d$, so since $d = 0$ for the four sample realizations in Figures \ref{fig-all realizations n = 10}, the plots of $\xi \mapsto l_D(\xi;d,\alpha)$ were all identical. The realized $P$ were .1513 and .1792, with corresponding plots of $\xi \mapsto l_P(\xi;p)$ being above zero for a small range of values of $\xi$, lending very mild support for $H_1$, but this support for $H_1$ disappears and reverses to support for $H_0$ when $\xi$ is increased. The intuition here is that when $\xi$ is large, it is expected that the values of $P$ should be much smaller than the values observed, but since they were not, then $H_0$ became more plausible than $H_1$ for the observed $p$. As such the information contained in the realization of $P$ is dependent on the alternative value under consideration and the direction of support it provides, whether towards $H_0$ or $H_1$, could even change with the same data.  Figure \ref{fig-all realizations n = 20}, with data generated under the null model with $n=20$, led to $d=0$ for both sample realizations, though for the second  realization, $p = .07$, which converted to a mild support for $H_1$ when the effect size is small, and support for $H_0$ when the effect size is large. In Figure \ref{fig-all realizations n = 20}, with data generated under the specific alternative model with $n=20$, both sample realizations led to $d = 1$, which are correct decisions; while the realized $P$ were about $\approx 0$ and .04. When $p \approx 0$, note that the plot of $p \mapsto l_P(\xi;p)$ is increasing over the range plotted, while when $p = .04$, it increases first, then starts to decrease and goes below zero, lending support for $H_0$. In these plots, we observe that the plot of $p \mapsto l_P(\xi;p)$ either decreases immediately, or it increases first then decreases. This is a manifestation of the mathematical result pointed out earlier that, even when $p$ is quite small, it could still lead to supporting $H_0$ if the effect size under consideration is quite large. This appears to be a fatal flaw of $P$ when its exact value is used to infer about the strength of support for either $H_0$ or $H_1$. This defect, however, is circumvented if one utilizes the value of $\rho^\prime(p)$ or $log(\rho^\prime(p)) = l_P(p)$ to deduce the strength of support for $H_0$ or $H_1$, or if the value $p$ is used in the MP decision function where an LoS $\alpha$ is specified and $H_0$ is rejected whenever $p \le \alpha$. In the latter case, whatever decision is made should be accompanied by a plot or table of values of $(\xi,l_D(\xi;d,\alpha))$. 

Next, we present an application to the famous R.\ A.\ Fisher's lady tea-tasting experiment, which ushered the era of NHST \cite{Fis71,Lyn19}. We discuss two versions of this tea-tasting experiment.

\medskip

{\bf Version 1:} There are eight cups on which tea has been mixed with milk. On each of these cups, the order in which the milk and the tea were placed is unknown to a lady -- a lady who claims that she is able to determine, with higher probability than just guessing, the order in which tea and milk were placed. Denote by $\theta$ the probability that the lady is able to correctly determine the order in which the tea and the milk were placed. The hypothesis testing problem is to test $H_0: \theta = .5$ versus $H_1: \theta > .5$, though from the NHST framework there is just the null hypothesis $H_0$ but no $H_1$.  Let $S$ denote the number of correct identifications by the lady. Then, under an independence assumption, $S$ has a binomial distribution with parameters $n=8$ and $\theta$, that is, the probability mass function of $S$ is
$$p_S(s|\theta) = \Pr\{S = s| \theta\} = {8 \choose s} \theta^s (1-\theta)^{8-s}, s = 0, 1, \ldots, 8.$$
The UMP size-$\alpha$ (randomized) decision function is
$$\delta^*(s,u;\alpha) = I\{S > c(\alpha)\} + I\{S = c(\alpha), u \le \gamma(\alpha)\}$$
with 
$$c(\alpha) = \inf\{c: \Pr\{S > c|\theta=.5\} \le \alpha\} \quad \mbox{and} \quad
\gamma(\alpha) = \frac{\alpha - \Pr\{S > c(\alpha)|\theta=.5\}}{\Pr\{S = c(\alpha)|\theta=.5\}}.$$
For the decision process $\Delta = \{\delta^*(\cdot,\cdot;\alpha): \alpha \in (0,1)\}$ the associated ROC function at $\theta = \theta_1$ is
$$\rho(\alpha;\theta_1) = \Pr\{S > c(\alpha) | \theta_1\} + \gamma(c(\alpha)) \Pr\{S = c(\alpha) | \theta_1\}$$
with derivative given by
$$\rho^\prime(\alpha;\theta_1) = \frac{\Pr\{S = c(\alpha) | \theta = \theta_1\}}{\Pr\{S = c(\alpha) | \theta = .5\}}.$$
This is the density function, which is piecewise constant, of the $P$-functional under $\theta = \theta_1$, with the $P$-functional being
$$P \equiv P(S,U) = \Pr\{S^* > S\} + U \Pr\{S^* = S\}$$
with $S^*$ having a binomial distribution with parameters $n=8$ and $\theta = .5$ and is independent of $(S,U)$.

\begin{figure}
\caption{Plots of $l_D(\theta_1;s)$ and $l_P(\theta_1;s)$ with respect to $\theta_1$ for different possible realizations of $S$, the number of correct identifications out of $n=8$ trials, in the tea-tasting experiment {\em under} the binomial model. For $l_D$, an LoS of $\alpha = .05$ was used. For the $l_D$ plot, the upper curve is for $s \ge 7$ with $d = 1$ ($H_0: \theta = .5$ was rejected), while the lower curve is for $s \le 6$ with $d = 0$ ($H_0$ not rejected).}
\label{fig-tea tasting binomial}
\includegraphics[width=\textwidth,height=2in]{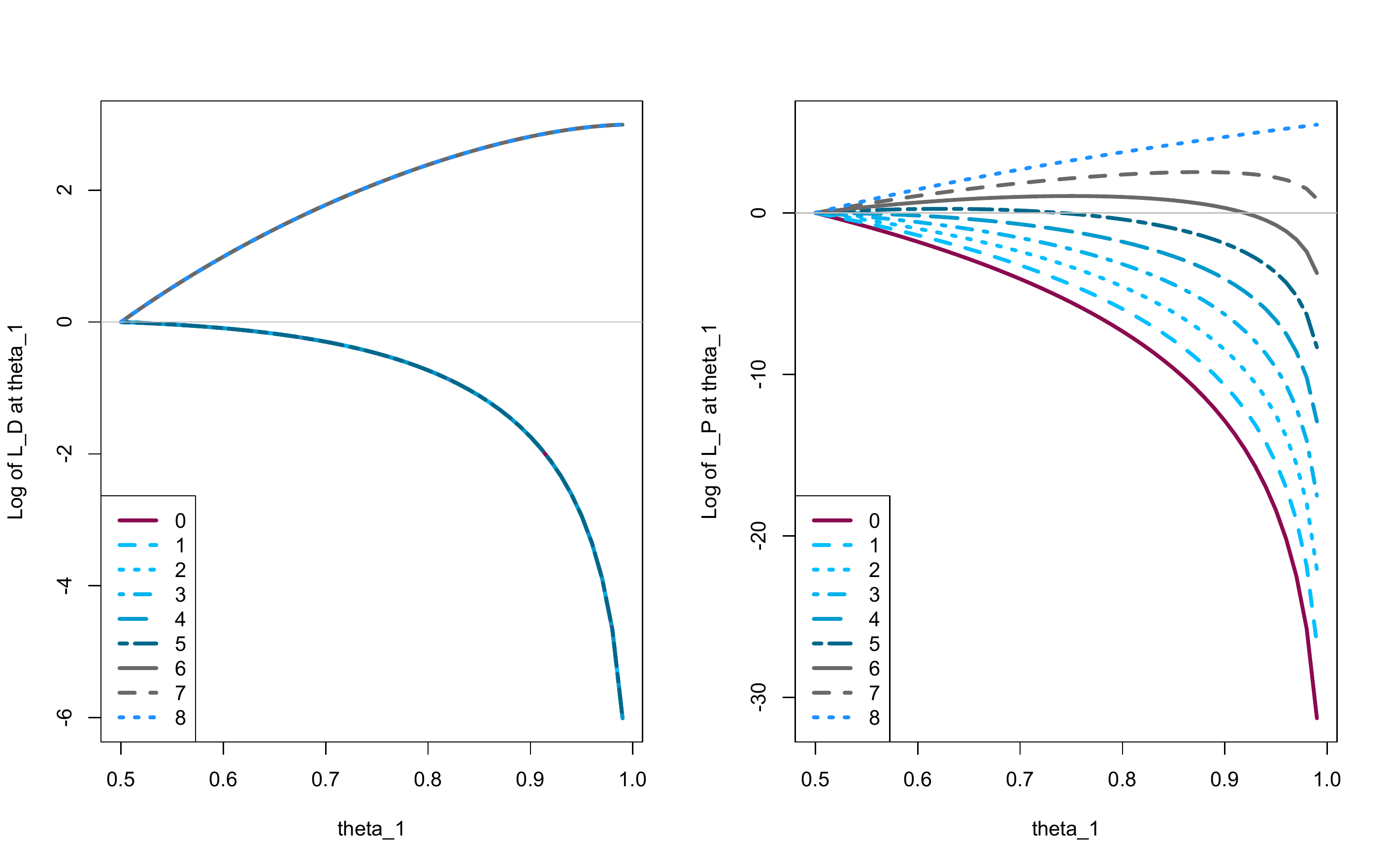}
\end{figure}

Figure \ref{fig-tea tasting binomial} provides the plots of $l_D(\theta_1;s)$ and $l_P(\theta_1;s)$ for different values of $\theta_1 > .5$ and for an LoS of $\alpha = .05$. The overlaid plots are for different possible realizations of $S$, which are $s \in \{0,1,2,\ldots,8\}$. In particular, focus on the case $s = 6$ which corresponds to the situation where the lady correctly identifies the order of placement of milk and tea in 6 of the 8 cups. The observed randomization value was $u = .973$ and the realized $P$ was $p = .1416$. The observed decision in this case is $d = 0$, so the $l_D$-plot is below zero, though still close to zero even for $\theta_1$ about .7,  but it then decreases rapidly as $\theta_1$ becomes larger. The associated $l_P$-plot is the third curve from the top, which is slightly above zero, but dips below zero when $\theta_1$ exceeds .9. Observe that even if $s = 8$, corresponding to the lady correctly identifying the orderings in all 8 cups, the strength of the support for $H_1$ is still not strong.

\medskip

{\bf Version 2:} This is the version described in \cite{Fis71} (see also \cite{Lyn19}). The lady is told that out of the 8 cups, 4 of them had the tea placed first before the milk, and in the other 4 it is in the reverse order. The lady is then asked to choose the 4 cups in which the tea was placed before the milk. Let $T$ be the number out of the four cups chosen by the lady in which tea preceded milk. Fisher introduced the notion of the null hypothesis which in this case coincides with the hypothesis that the lady does not have discriminatory abilities, hence the distribution of $T$ under this hypothesis is hypergeometric with parameters 4, 4, and 4, that is,
\begin{equation}
\label{hypergeometric}
p_T(t) = \Pr\{T = t|H_0\} = \frac{{4 \choose t}{4 \choose {4-t}}}{{8 \choose 4}}, t = 0,1,2,3,4.
\end{equation}
$H_0$ is then rejected when $T$ is large, though even this rejection criterion tacitly considers the alternative model that the lady has discriminatory powers. But one could formulate this problem more formally to accord with the Neyman-Pearson framework. When $\theta$ is the probability that the lady could determine which came first (tea or milk), then the probability mass function of $T$ is given by
\begin{equation}
\label{dist of T under H1}
p_T(t;\theta) = \Pr\{T = t|\theta\} = \frac{{4 \choose t}^2 \theta^t (1-\theta)^{8-t}}{\sum_{j=0}^4 {4 \choose j}^2 \theta^j (1-\theta)^{8-j}}, t = 0,1,2,3,4.
\end{equation}
Observe that when $\theta = .5$, $p_T(t;.5)$ is the hypergeometric distribution given in (\ref{hypergeometric}). As in the binomial model, the UMP decision function of size $\alpha$ is of the same form with $T$ replacing $S$ and with $p_T(t;\theta)$ replacing $p_S(s;\theta)$. 

\begin{figure}
\caption{Plots of $l_D(\theta_1;s)$ and $l_P(\theta_1;s)$ with respect to $\theta_1$ for different possible realizations of $T$, the number of correct identifications out of the four cups chosen by the lady in the tea-tasting experiment as described in \cite{Fis71}. For $l_D$, an LoS of $\alpha = .05$ was used. For the $l_D$ plot, the upper curve is for $t = 4$ with $d = 1$ ($H_0: \theta = .5$ was rejected), while the lower curve is for $t \le 3$ with $d = 0$ ($H_0$ not rejected). In Fisher's experiment, the observed value of $T$ was $t = 3$.}
\label{fig-tea tasting hypergeometric}
\includegraphics[width=\textwidth,height=2in]{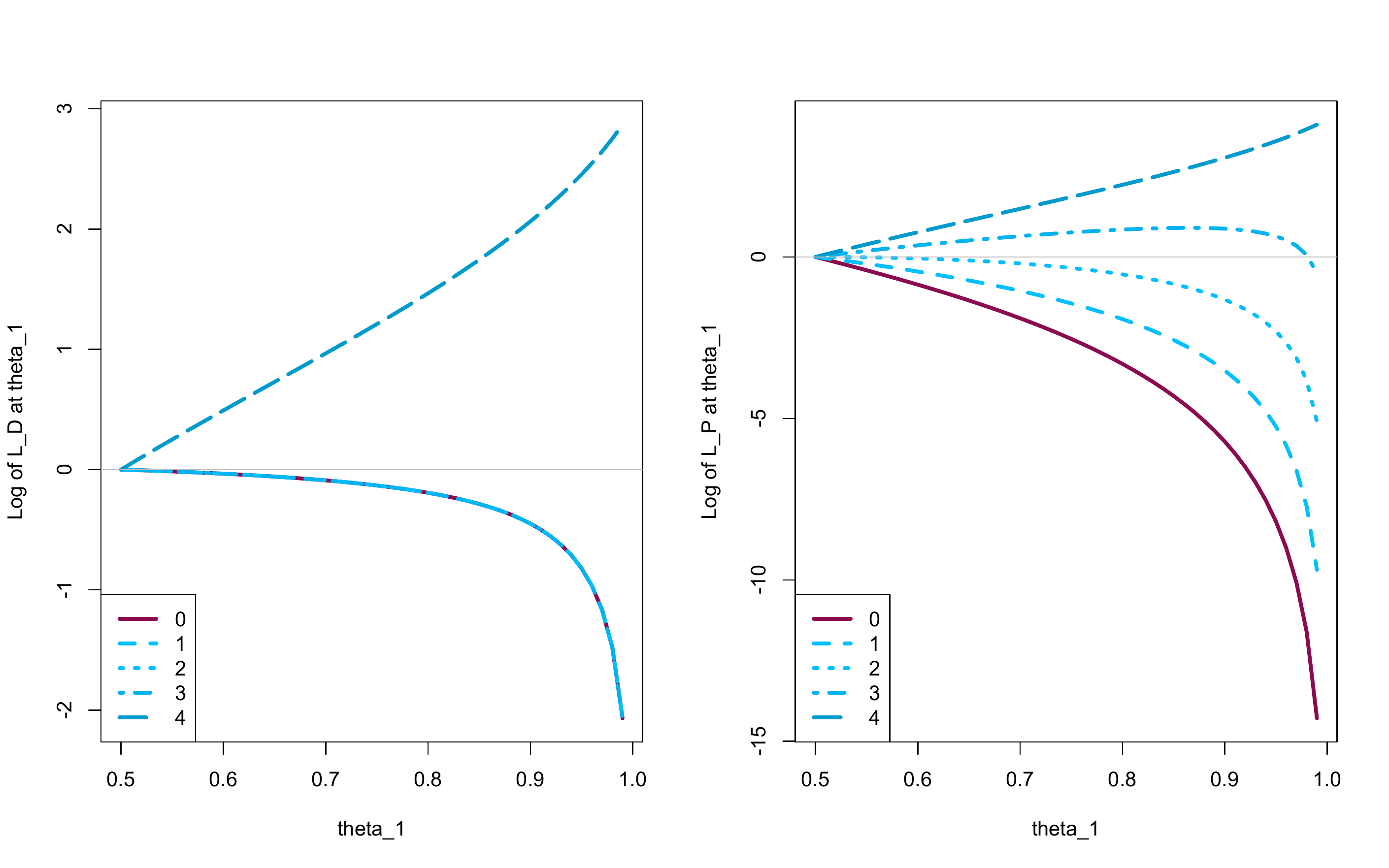}
\end{figure}

The profiles obtained for this version of the experiment are similar to the binomial version as shown in Figure \ref{fig-tea tasting hypergeometric}. For the observed value of $t = 3$ and with randomization value of $u = .815$, the decision is not to reject $H_0$ at LoS of $\alpha = .05$ and the realized $P$ is $p = 0.2007$. This result could be considered as inconclusive since we note that the associated $l_D(\theta_1)$ is close to zero. The same conclusion arises by looking at the second curve from the top in the $l_P(\theta_1)$ plot which is also close to zero. {\em If} the lady had instead obtained $t = 4$, which would have coincided with the correct identifications for all 8 cups, then there would have been a stronger support for $H_1$ based on these profile plots.

\section{Sequential Learning}
\label{sec-Sequential Learning}

The Bayes updating formulas in (\ref{update on x}), (\ref{update on d}), and (\ref{update on p}) when given $x$, $d$, or $p$, respectively, can be used to sequentially update the knowledge about $H_0$ and $H_1$ as results of other studies are sequentially obtained. The current posterior probabilities will become the prior probabilities at the beginning of the next study, and the result of this new study, whether $x$, $d$, or $p$, will be used to update these prior probabilities. The beauty of this approach is that whatever study is performed, so long as it is performed with integrity and honesty, whether it is a study with large or small sample sizes, will be able to contribute to the updating of the knowledge about $H_0$ and $H_1$. Theorem \ref{posterior convergence} states that if $H_0$ [$H_1$] is true, and provided that the initial prior probabilities of $H_0$ and $H_1$ are not 0 nor 1, then the sequence of posterior probabilities for $H_0$ [$H_1$] will converge almost surely to 1. As such, one could specify a threshold $\epsilon > 0$, say $\epsilon = .0001$ or determined in considerations of the cost consequences of decisions, such that when either posterior probability exceeds $1 - \epsilon$, a final conclusion on the truth or falsity of $H_0$ and $H_1$ is made. This sequential approach to learning about the truth of $H_0$ and $H_1$ is depicted in the pseudo-code below. 

\medskip

\begin{center}
{\bf Pseudo-Code for Sequential Updating of Knowledge of $H_0$ and $H_1$}
\end{center}

\noindent
{\tt START;} \newline
Specify Threshold, $\epsilon = 0.0001$ (say); \newline
Specify Prior Probability of $H_0$, $\kappa_0 \in (0,1)$; \newline
{\tt L1:} \newline
\indent Decide on Experiment and the Decision Process $\Delta = \{\delta(x,u;\alpha)\}$ to Use; \newline
\indent Specify LoS, $\alpha \in (0,1)$; \newline
\indent Perform Experiment to Obtain Data $(x,u)$; \newline
\indent Compute $d = \delta(x,u;\alpha)$ or $p = P(x,u)$; \newline
\indent Update Probability of $H_0$ into $\kappa_0(x,u)$ by either using $(x,u)$, $d$, or $p$ according to updating formulas (\ref{update on x}), (\ref{update on d}), or (\ref{update on p}), respectively; \newline
\indent If (($k_0(x,u) < \epsilon$) or ($k_0(x,u) > 1-\epsilon$)) then \newline
\indent\indent \{Declare $H_0$ is TRUE [FALSE] if $\kappa_0(x,u) > 1 - \epsilon$ [$\kappa_0(x,u) < \epsilon$]; {\tt STOP.}\} \newline
\indent Else \{$\kappa_0 \leftarrow \kappa_0(x,u)$; Back to {\tt L1};\}

\medskip


We will now establish Theorem \ref{posterior convergence}. We do so by establishing some intermediate results. We will assume the following set of conditions:

\medskip

\begin{center}
\begin{minipage}{\textwidth}
\begin{center}
{\bf Set of Conditions}
\end{center}
\begin{itemize}
\item[(i)] Goal is to decide or gain knowledge about $H_0: f = f_0$ and $H_1: f = f_1$.
\item[(ii)] $\nu(\{x: f_0(x) \ne f_1(x)\}) > 0$.
\item[(iii)] For the $m$th study ($m = 1, 2, \ldots$), $X_m = (X_{mj}, j=1,2,\ldots,n_m)$, which are IID random variables from $f$, are observed.
\item[(iv)] The observables $X_1, X_2, \ldots, X_m, \ldots$ are independent of each other.
\item[(v)] The initial prior probabilities $\kappa_0$ and $\kappa_1 = 1 - \kappa_0$ are in $(0,1)$.
\end{itemize}
\end{minipage}
\end{center}

We denote by $\delta_m^*(\cdot,\cdot;\alpha_m)$ the size-$\alpha_m$ MP decision function and by $\rho_m(\alpha_m)$ its associated (power) ROC function. The ROC function of the decision function based on only one observation will be denoted by $\bar{\rho}(\cdot)$. We also denote by $$\Lambda_m(x_m) = \prod_{j=1}^{n_m} \frac{f_1(x_{mj})}{f_0(x_{mj})} = \prod_{j=1}^{n_m} \Lambda(x_{mj})$$ the likelihood ratio for the $m$th study, with $\Lambda(x) = f_1(x)/f_0(x)$. 

\begin{lemma}
\label{lemma-ROC domination}
$\alpha \mapsto \rho_m(\alpha)$ satisfies $\rho_m(\alpha) \ge \bar{\rho}(\alpha) > \alpha$ for every $\alpha \in (0,1)$.
\end{lemma}

\begin{proof}
The second inequality has already been established in Proposition \ref{prop-ROC properties}. The first inequality follows from the fact that the size-$\alpha$ MP decision function based on only one observation belongs to the collection of all size-$\alpha$ decision functions based on $n_m$ observations. As such the power of the MP decision function in this collection is at least equal to all other decision functions in this collection, in particular, the MP decision function based only on one observation.
\end{proof}

\begin{lemma} 
\label{lemma-convergence of sum of log Lambda}
If $E_0[|\log \Lambda(X)|] < \infty$, then as $M \rightarrow \infty$,
$\sum_{m=1}^M \log \Lambda_m(X_m) \rightarrow -\infty$
almost surely under $H_0$.
\end{lemma}

\begin{proof}
First note that
$\sum_{m=1}^M \log \Lambda_m(X_m) =  \sum_{m=1}^M \sum_{j=1}^{n_m} \log \Lambda(X_{mj}).$
Under $H_0$ we have, using Jensen's Inequality and since $\Pr_0\{f_0(X) \ne f_1(X)\} > 0$, that
\begin{equation}
\label{negative expectation of log Lambda}
E_0[\log \Lambda(X)] < \log E_0[\Lambda(X)] = \log \int \left[\frac{f_1(x)}{f_0(x)}\right] f_0(x) \nu(dx) = \log (1) = 0.
\end{equation}
Since $\{\log \Lambda(X_{mj}), j=1,\ldots,n_m; m=1,2,\ldots\}$ are IID random variables and by the condition that $E_0[\log \Lambda(X)] < \infty$, then by Kolmogorov's Strong Law of Large Numbers (see Theorem 7.5.1 in \cite{Res14}) it follows that, as $M \rightarrow \infty$,
$$\frac{1}{\sum_{m=1}^{M} n_m} \sum_{m=1}^M \sum_{j=1}^{n_m} \log \Lambda(X_{mj})$$
converges almost surely, under $H_0$, to $E_0[\log \Lambda(X)]$, which from (\ref{negative expectation of log Lambda}) is strictly negative. The result thence follows since $n_m \ge 1$ for every $m$ so that $\sum_{m=1}^M n_m \rightarrow \infty$ as $M \rightarrow \infty$.
\end{proof}

Below we let, for $m = 1,2, \ldots$,
$D_m = \delta_m^*(X_m,U_m;\alpha_m)$ and
$$V_m = D_m \log \left[\frac{\rho_m(\alpha_m)}{\alpha_m}\right] + (1-D_m) \log\left[\frac{1-\rho_m(\alpha_m)}{1-\alpha_m}\right].$$

\begin{lemma}
\label{convergence of l_D}
Suppose there exists an $\epsilon > 0$ and an open interval $I \subset [0,1]$ such that
$$\sup_{\alpha \in I} \left[ \alpha \log \left(\frac{\bar{\rho}(\alpha)}{\alpha}\right) + (1-\alpha) \log \left(\frac{1-\bar{\rho}(\alpha)}{1 - \alpha}\right) \right] \le -\epsilon,$$
and for the sequences of LOSs $\{\alpha_m, m=1,2,\ldots\} \subset I$ and sample sizes $\{n_m, m=1,2,\ldots\}$, we have
$$\sum_{m=1}^\infty \frac{\alpha_m(1-\alpha_m)}{m^2} \left[\log\left(\frac{\rho_m(\alpha_m)}{\alpha_m} \frac{1-\alpha_m}{1-\rho_m(\alpha_m)}\right)\right]^2 < \infty.$$
Then,  under $H_0$,
$\sum_{m=1}^M V_m \rightarrow -\infty$,
almost surely as $M \rightarrow \infty$.
\end{lemma}

\begin{proof}
It is easy to see that the mapping $y \mapsto g(y) = \alpha \log(y/\alpha) + (1-\alpha) \log((1-y)/(1-\alpha))$ for $\alpha \in (0,1)$ and $y \in (\alpha,1)$ is decreasing in $y$. Consequently, using this result, Lemma \ref{lemma-ROC domination}, and the first condition of the Lemma,
\begin{eqnarray*}
\mu_m & \equiv & E_0(V_m) = \alpha_m \log\left(\frac{\rho_m(\alpha_m)}{\alpha_m}\right) + (1-\alpha_m) \log\left(\frac{1-\rho_m(\alpha_m)}{1 - \alpha_m}\right) \\
& \le & \alpha_m \log\left(\frac{\bar{\rho}(\alpha_m)}{\alpha_m}\right) + (1-\alpha_m)\log\left(\frac{1 - \bar{\rho}(\alpha_m)}{1-\alpha_m}\right) \\ & \equiv & \bar{\mu}(\alpha_m) 
 \le  \sup_{\alpha \in I} \bar{\mu}(\alpha) 
 \le  -\epsilon.
\end{eqnarray*}
By the second condition of the Lemma, we have
$\sum_{m=1}^\infty Var_0\left[{V_m}/{m}\right] < \infty.$
Therefore, by Corollary 7.4.1 in \cite{Res14}, we conclude that, under $H_0$,
$$\frac{1}{M} \sum_{m=1}^M {(V_m - \mu_m)} \rightarrow 0$$ almost surely. Thus, as $M \rightarrow \infty$, since $\sum_{m=1}^M \mu_m \rightarrow -\infty$, then $\sum_{m=1}^M V_m \rightarrow -\infty$ almost surely.
\end{proof}

We remark that the conditions in the lemma impose some constraints on the sample size and the LoS sequences. These conditions will hold, for instance, if the sample sizes $n_m$s are bounded and $\alpha_m$s are bounded away from 0 and 1, which in practice are realistic conditions.

\begin{lemma}
\label{lemma-convergence of l_P}
Let $P_m$ be the $P$-functional for the $m$th study.  If
$$\sum_{m=1}^\infty \frac{1}{m^2} \int_0^1 [\log \rho_m^\prime(u)]^2 du < \infty$$
and, for some $\epsilon > 0$,
$\limsup_m \int_0^1 \log \rho_m^\prime(u) du \le -\epsilon$
then, under $H_0$, almost surely, 
$$\sum_{m=1}^M \log \rho_m^\prime(P_m) \rightarrow -\infty.$$
\end{lemma}

\begin{proof}
The first condition guarantees the finiteness of the mean and variance of $\log \rho_m^\prime(P_m)$, and by Jensen's Inequality we have
$$E_0[\log \rho_m^\prime(P_m)]  < \log E_0\rho_m^\prime(P_m) = \log \int_0^1 \rho_m^\prime(u) du = \log (1) = 0.$$
The second condition guarantees that $\sum_{m=1}^M E_0[\log \rho_m^\prime(P_m)] \rightarrow -\infty$ as $M \rightarrow \infty$. Since the first condition also implies that
$$\sum_{m=1}^\infty Var_0\left[\frac{\log \rho_m^\prime(P_m)}{m}\right] < \infty,$$
then, by Corollary 7.4.1 in \cite{Res14},  we have, as $M \rightarrow \infty$,
$$\sum_{m=1}^M \frac{(\log \rho_m^\prime(P_m) - E_0[\log \rho_m^\prime(P_m)])}{M} \rightarrow 0$$ almost surely under $H_0$. Thus, as $M \rightarrow \infty$, since $\sum_{m=1}^M E_0[\log \rho_m^\prime(P_m)] \rightarrow -\infty$, then $\sum_{m=1}^M \log \rho_m^\prime(P_m) \rightarrow -\infty$ almost surely under $H_0$.
\end{proof}

\begin{theorem}
\label{posterior convergence}
Under the conditions stated, together with the conditions in the lemmas, the sequence of posterior probabilities of $H_0$, given either $X_m$, $D_m$, or $P_m$ on the $m$th study, converges almost surely, under $H_0$, to one.
\end{theorem}

\begin{proof}
After the $m$th study, the posterior probability of $H_0$ is given by
\begin{eqnarray*}
\kappa_{0m} & = & \left[1 + \frac{\kappa_1}{\kappa_0} \times  \exp\left\{
\sum_{j=1}^m I\{J_j =1\} \log \Lambda_j(X_j) + 
\right.\right. \\
&& \left.\left.
\sum_{j=1}^m I\{J_j = 2\} V_j +
\sum_{j=1}^m I\{J_j = 3\} \log \rho_j^\prime(P_j)\right\}\right]^{-1},
\end{eqnarray*}
where $J_m$ equals 1 if $X_m$ is reported, 2 if $D_m$ is reported, and 3 if $P_m$ is reported in the $m$th study. Let $\{J_m, m=1,2,\ldots\}$ be the specific sequence that is observed. Define subsequences
$\{J_m^{(1)}, m=1,2,\ldots\}$, $\{J_m^{(2)}, m=1,2,\ldots\}$, and $\{J_m^{(3)}, m=1,2,\ldots\}$
with $J_m^{(j)} = \min\{k > J_m^{(j)}: J_m = j\}$ for $j=1,2,3$. At least one of these subsequences will be an infinite subsequence (note that it is possible all three subsequences are infinite subsequences). By Lemmas \ref{lemma-convergence of sum of log Lambda}, \ref{convergence of l_D}, and \ref{lemma-convergence of l_P}, outside of a $\Pr_0$-null set, $\{\sum_{j=1}^m \log \Lambda_j(X_j), m=1,2,\ldots\}$, $\{\sum_{j=1}^m V_j, m=1,2,\ldots\}$, and $\{\sum_{j=1}^m \log \rho_j^\prime(P_j)\}$ all diverges to $-\infty$, hence any {\em infinite} subsequence from each of these sequences will also diverge to $-\infty$. Since at least one of $\{J_m^{(j)}, m=1,2,\ldots\}$ is an infinite subsequence, then outside of a $P_0$-null set, $\{\sum_{j=1}^m I\{J_j =1\} \log \Lambda_j(X_j) + \sum_{j=1}^m I\{J_j = 2\} V_j + \sum_{j=1}^m I\{J_j = 3\} \log \rho_j^\prime(P_j), m=1, 2, \ldots\}$ diverges to $-\infty$. Therefore, $P_0$-almost surely, $\{\kappa_{0m}, m=1,2,\ldots\}$ converges to 1 since $\kappa_0 \in (0,1)$.
\end{proof}

To demonstrate the phenomenon described in Theorem \ref{posterior convergence} empirically, we refer back to the fourth panels of Figures \ref{under H0} and \ref{under H1}. These panels depict the sequence of posterior probabilities for $H_0$ updated from the sequence of decisions $d_m$s and also the sequence of realized $P$-functionals $p_m$s. Observe that the plot of these posterior probabilities with respect to $m$ are wiggly for small $m$, but starts stabilizing and converging to 1 when $H_0$ is true and to 0 when $H_0$ is false. Through sequential updating, the issue of whether a specific result is replicable or not becomes a moot point, since whatever that result, so long as it was obtained with integrity and honesty, it will contribute to the updating of the knowledge about $H_0$ and $H_1$, in essence, it becomes a data point for the search for truth.

%

\section{Impact of Publication or Reporting Bias}
\label{sect-Publication bias}

Theorem \ref{posterior convergence} points to a coherent and sensible approach to sequentially acquiring knowledge about $H_0$ and $H_1$ in the sense that if more and more studies are conducted, then the sequence of posterior probabilities of $H_0$ [$H_1$] will converge to 1 under $H_0$ [$H_1$]. But, what could potentially derail this approach? The reality of the publication process of manuscripts submitted for publication or just on the reporting of results of studies is the existence of a bias against results that are `not-significant'. This means that if a study resulted in a decision $d = 0$, then there is a smaller chance that this result will be published or reported compared to when the decision is $d = 1$. Similarly, studies with large observed $P$-functionals have smaller chances of getting reported or published. Such studies whose results do not get published or reported could therefore not be used in the posterior probability updating of the knowledge about $H_0$ and $H_1$. What will be the impact of this?

Recall that we denoted by
$$V = D \log (\rho(\alpha)/\alpha) + (1-D) \log((1-\rho(\alpha))/(1-\alpha))$$
with $D = \delta^*(X,U;\alpha)$ the log-likelihood ratio based on realization of the decision function. When there is no bias in the publication or reporting process, then, under $H_0$, $D$ has a Bernoulli distribution with success probability of $\alpha$. We have used this result to show that, under $H_0$, by using Jensen's Inequality, we have $E_0(V) < 0$. This is the reason why, under $H_0$, the sequence of posterior probabilities converges almost surely to 1.

But consider a simple model where a publication bias exists. Thus, upon observing $D$, which is governed by the probability $P_0$ (more precisely, $P_0 \otimes \lambda$), there is a second stochastic step which determines whether the result of the study will be reported or not. This step will depend only on the observed value of $D$. Let $I = 1 (0)$ if the result of the study is published or reported (not published nor reported). We suppose that the conditional probability distribution of $I$, given $D$, is given by
\begin{eqnarray*}
& Q(I=1|D=1) = \eta_1 = 1 - Q(I=0|D=1);& \\
& Q(I=1|D=0) = \eta_0 = 1 - Q(I=0|D=0), &
\end{eqnarray*}
where $0 \le \eta_0 < \eta_1 \le 1$. We denote by $P_0^*$ the probability measure that governs $(D,I)$ under $P_0$. Under this model, we will only be able to observe $D$ if $I=1$. Thus, the observable $D^*$, which equals $D$ when $I=1$, possesses the probability distribution, under $H_0$, given by
$$P_0^*(D^* = 1) = P_0^*(D=1|I=1) = \frac{\alpha \eta_1}{\alpha \eta_1 + (1-\alpha) \eta_0} = 1 - P_0^*(D^*=0).$$
Observe that since $\eta_0 < \eta_1$, we have $P_0^*(D^* = 1) > \alpha$, indicating that the desired size of $\alpha$ is not anymore satisfied if $D^*$ is the variable that is observed instead of $D$. Denote by
$$V^* = D^* \log (\rho(\alpha)/\alpha) + (1-D^*) \log((1-\rho(\alpha))/(1-\alpha))$$
the observable log-likelihood ratio, purportedly based on $D$.
With $E_0^*(\cdot)$ denoting expectation with respect to the probability measure $P_0^*$, we have
\begin{eqnarray*}
E_0^*(V^*) & = & E_0^*\{D^* \log(\rho(\alpha)/\alpha) + (1-D^*) \log((1-\rho(\alpha))/(1-\alpha))\} \\
& = & \left(\frac{\alpha\eta_1}{\alpha\eta_1 + (1-\alpha)\eta_0}\right) \log\left[
\frac{\rho(\alpha) (1-\alpha)}{\alpha (1-\rho(\alpha))}\right] + \log\left[\frac{1-\rho(\alpha)}{1-\alpha}\right] \\
& >  & \alpha \log\left[
\frac{\rho(\alpha) (1-\alpha)}{\alpha (1-\rho(\alpha))}\right] + \log\left[\frac{1-\rho(\alpha)}{1-\alpha}\right] \\
& = & E_0(V).
\end{eqnarray*}
We already know that $E_0(V) < 0$, but because of the above inequality, we could not anymore guarantee that $E_0^*(V^*) < 0$, that is, it could be that $E_0^*(V^*) > 0$. An extreme case of this situation is when $\eta_1 = 1$ and $\eta_0 = 0$, where studies that do not lead to rejections of $H_0$ are never getting published nor reported, while those that lead to rejections of $H_0$ are always getting published or reported. In this extreme case we will have
$$E_0^*(V^*) = \log\left[
\frac{\rho(\alpha) (1-\alpha)}{\alpha (1-\rho(\alpha))}\right] + \log\left[\frac{1-\rho(\alpha)}{1-\alpha}\right] =  \log\left[
\frac{\rho(\alpha)}{\alpha}\right] > 0.$$
Note that this situation is not outside the realm of possibilities. Imagine, for instance, the situation where a well-respected research team, headed by a prominent Nobel-prize winning scientist, have just published a study which led to the rejection of $H_0$. Subsequent studies by other research teams that did not lead to the rejection of the same $H_0$ may then never see the light of day -- for how could they go against the established research team? But, it is possible that the well-respected research team has committed an error of decision in rejecting $H_0$, since there are no certainties in decision-making based on data.
Now, in such cases where $E_0^*(V^*) > 0$, there is the possibility that, {\em even under $H_0$}, the sequence of posterior probabilities of $H_0$, based on a sequence of $V_1^*, V_2^*, \ldots$, may converge to 0, instead of 1.

Let us also examine the case with $P$-functionals. We have shown, using Jensen's Inequality and the uniformity of $P$ under $H_0$, that $E_0[\log \rho^\prime(P)] < 0$. In the presence of publication or reporting bias against large $P$'s, with $I$ indicating once more if the result of a study is published or reported, suppose that
$$Q(I=1|P=p) = g(p), p \in (0,1),$$
with $g(\cdot)$ not identically equal to 1 almost everywhere and $g(p)$ is non-increasing in $p$. Note that $g(p) \le 1$ hence $\int_0^1 g(p) dp < 1$. These conditions imply that smaller values of $P$ lead to higher chances of getting published or reported. Let us denote by $P^*$ the published or reported $P$, and denote by $P_0^*$ the probability measure induced by $P_0$ and $Q$ that governs $(P,I)$ under $H_0$. Then, the probability density function of $P^*$, under $H_0$, is given by
$$h_{P^*}(p) = \frac{(1)g(p)}{\int_0^1 (1) g(w) dw} \equiv \bar{g}(p), p \in (0,1).$$
Since $g(\cdot)$ is non-increasing, then, for every $p \in (0,1)$, $\int_0^p \bar{g}(w) dw \ge \int_0^p (1) dw = p$, with strict inequality for some $p$. Thus, under $H_0$, $P^*$ is stochastically smaller than $P$. Since $p \mapsto \log \rho^\prime(p)$ is non-increasing, then it follows that
$$E_0^*[\log \rho^\prime(P^*)] = \int_0^1 [\log\rho^\prime(p)] \bar{g}(p) dp > \int_0^1 [\log\rho^\prime(p)] (1) dp = E_0[\log\rho^\prime(P)].$$
Note that this result also follows from the fact that if $V_1 \stackrel{st}{<} V_2$ and if $a(\cdot)$ is an integrable non-increasing function, then $E[a(V_1)] > E[a(V_2)]$. Thus, even if $E_0[\log\rho^\prime(P)] < 0$, there is no more guarantee that $E_0^*[\log\rho^\prime(P^*)] < 0$. In fact, consider an extreme case where we take $g(p) = I\{p \le c\}$ where $c \in (0,1)$ is such that $\rho^\prime(p) >(<) 1$ when $p <(>) c$. Then, in this case, we will have
$$E_0^*[\log\rho^\prime(P^*)] = \int_0^1 [\log\rho^\prime(p)] \bar{g}(p) dp = \frac{1}{c} \int_0^c [\log\rho^\prime(p)] dp > 0.$$
Thus, in the presence of publication or reporting bias, we could have  that $$E_0^*[\log\rho^\prime(P^*)] > 0,$$ so that it is possible that the sequence of posterior probabilities of $H_0$ updated from a sequence of $P_1^*, P_2^*, \ldots$ may converge to 0, instead of 1, {\em even under $H_0$.} 

\section{Considerations on the Choice of LoS}
\label{sect-Choice of LoS}

If possible, it would be preferable to adhere to the principle that a mathematical theory of statistical decision-making or of updating knowledge about $H_0$ and $H_1$ should {\em not} rely on arbitrary constants which cannot be justified within the theory --- for if it does the theory becomes mathematically `impure'. The rationale behind the conventional choice of $\alpha = 0.05$, or some other small value such as 0.01 or 0.10, for the LoS in existing statistical decision-making procedures is in order to put an upper bound to the probability of committing a Type I error, which occurs when $H_0$ is rejected when it is in fact true, considered to be a more serious type of error than a Type II error, which is accepting $H_0$ when in fact it is false. However, this arbitrary choice of $\alpha=0.05$ is difficult to justify within the theory and it could be viewed as a stain in existing statistical decision-making methods and perhaps could be their Waterloo. This choice has been the source of controversies and criticisms of existing hypothesis testing methods. Criticizing the use of such a threshold, Professor John Ioannidis \cite{Ion05} wrote:

\begin{quote}
... the high rate of nonreplication (lack of confirmation) of research discoveries is a consequence of the convenient, yet ill-founded strategy of claiming conclusive research
findings solely on the basis of a single study assessed by formal statistical significance, typically for a $p$-value less than 0.05.
\end{quote}

The attacks on $P$-values and the LoS threshold of $\alpha = .05$ has in fact continued unabated. Recently, there was the editorial article by Wasserstein, Schirm and Lazar \cite{WasSchLaz19} with the provocative title: {\em Moving to a {W}orld {B}eyond ``$p < 0.05$''}. See also the many articles in {\em The American Statistician} issue in which \cite{WasSchLaz19} appeared, which deal with $P$-value controversies, criticisms, and proposed changes. The paper \cite{SzuIoa17} also discusses reasons why the NHST approach is an unsuitable tool in scientific research. But, see also the papers of \cite{Ber03,Chr05} that provide comparisons, contrasts, and reconciliations of the NHST, Neyman-Pearson, and Bayesian approaches to statistical decision-making.

So, how should one choose an LoS value to use? As pointed out earlier, any decision-maker could choose any LoS he desires since whatever his choice is will be properly taken into account in the log-likelihood ratio based on $d$. So, the appropriate question is whether there is a way to choose an LoS value to optimize the information that the decision-maker could acquire from his study. We will argue below that it should depend on the major goal of the decision-maker and the viewpoint of his decision-making process.

\subsection{Game-Theoretic Approach}
\label{subsec-minimax}

If the decision-maker desires to make an immediate decision between $H_0$ and $H_1$ based on his study, then he should consider the cost consequences of his possible decisions. Thus, suppose that $C_{ij}$ is the cost incurred by deciding for $H_j$ when $H_i$ is the truth, with $C_{01} > C_{00}$ and $C_{10} > C_{11}$. The specification of the cost consequences of decisions should reflect the relative severities of the possible decisions under the two possible states of reality. For the decision function $\delta^*(\alpha)$, the expected costs or risks under $H_0$ and $H_1$ are, respectively,
\begin{eqnarray*}
R_0(\alpha) = E_0[\mbox{Cost}(\alpha)] & = & C_{00} (1-\alpha) + C_{01} \alpha; \\
R_1(\alpha) = E_1[\mbox{Cost}(\alpha)] & = & C_{10} (1-\rho(\alpha)) + C_{11} \rho(\alpha).
\end{eqnarray*}
If the decision-maker supposes that he is making his decision against an adversary who controls the choice of which of $H_0$ or $H_1$ will be the truth and whose intent is to maximize his cost, or if he simply wants to be as conservative as possible, then he should utilize an LoS $\alpha$ which will minimize his maximum risk, a minimax approach. The appropriate $\alpha_{M}$ is then given by
\begin{equation}
\label{MinMax solution}
\alpha_{M} = {\arg\min}_{\alpha \in [0,1]} \max\{R_0(\alpha), R_1(\alpha)\}.
\end{equation}
If the curves $\alpha \mapsto R_0(\alpha)$ and $\alpha \mapsto R_1(\alpha)$ intersect over $\alpha \in [0,1]$, then the value of $\alpha_M$ is the $\alpha$ satisfying $R_0(\alpha) = R_1(\alpha)$. There will be a point of intersection if $C_{10} \ge C_{00}$ and $C_{01} \ge C_{11}$ (for instance, when $C_{00} = C_{11} = 0$), and $\alpha_M$ satisfies the equation
\begin{equation}
\label{equation minmax}
\rho(\alpha_M) + \left[\frac{C_{01}-C_{00}}{C_{10}-C_{11}}\right] \alpha_M - \left[\frac{C_{10}-C_{00}}{C_{10}-C_{11}}\right] = 0.
\end{equation}
On the other hand, if $C_{11} \ge C_{01}$, then $\alpha_{M} = 1$; while if $C_{00} \ge C_{10}$, then $\alpha_{M} = 0$.

\subsection{Bayes Approach}
\label{subsec-Bayes}

An alternative for the decision-maker is to define a weighted expected cost or risk, where the risks under $H_0$ and $H_1$ are weighted by the respective prior probabilities of $H_0$ and $H_1$, resulting in the Bayes risk for $\delta^*(\alpha)$, given by
$$\mbox{BayesRisk}(\alpha) = \kappa_0 R_0(\alpha) + (1-\kappa_0) R_1(\alpha).$$
The decision-maker could then choose the $\alpha$ minimizing $\mbox{BayesRisk}(\alpha)$, denoted by $\alpha_B$. Here, $\kappa_0$ represents the decision-maker's prior probability that $H_0$ is true.  The solution is the $\alpha_B$ that satisfies
\begin{equation}
\label{Bayes solution}
\rho^\prime(\alpha_B) = \frac{\kappa_0}{1-\kappa_0} \frac{C_{01}-C_{00}}{C_{10} - C_{11}}.
\end{equation}
This choice of LoS, denoted by $\alpha_B$, leads to the decision function $\delta^*(\alpha_B)$ which coincides with the Bayes decision function, which is the one minimizing the Bayes risk among all decision functions (see, for instance, \cite{Sch95}). An indirect argument for this result is that Bayes decision functions are also of the form of the MP decision functions, hence since we minimized the Bayes risk among the MP decision functions, the $\alpha_B$ in (\ref{Bayes solution}) therefore finds the Bayes decision function.

\subsection{Optimizing Knowledge Accrual}
\label{subsec-change}

Another alternative for the decision-maker is to focus instead on optimizing the change that will accrue about his knowledge of $H_0$ or $H_1$ based on the result of his decision function. It is then sensible to choose an LoS $\alpha$ such that the expected log-likelihood ratios, based on $D$, under $H_0$ and $H_1$, achieve their maximum difference. The expected log-likelihood ratio based on $D$, under $H_0$ and $H_1$, are respectively,
\begin{eqnarray*}
E_0[\log \Lambda_D] & = & \alpha \log\left[\frac{\rho(\alpha)}{\alpha}\right] + (1-\alpha) \log\left[\frac{1-\rho(\alpha)}{1-\alpha}\right]; \\
E_1[\log \Lambda_D] & = & \rho(\alpha) \log\left[\frac{\rho(\alpha)}{\alpha}\right] + (1-\rho(\alpha)) \log\left[\frac{1-\rho(\alpha)}{1-\alpha}\right].
\end{eqnarray*}
The difference between these two expected log-likelihood ratios is
\begin{displaymath}
\mathcal{D}(\alpha) = E_1[\log \Lambda_D] - E_0[\log \Lambda_D] = [\rho(\alpha) - \alpha] \log\left[\frac{\rho(\alpha)}{\alpha} \frac{1 - \alpha}{1 - \rho(\alpha)}\right].
\end{displaymath}
The decision-maker could then choose the $\alpha$ that maximizes $\mathcal{D}(\alpha)$, that is,
\begin{equation}
\label{best discrimination}
\alpha_{D} = {\arg\max}_{\alpha \in [0,1]} \mathcal{D}(\alpha).
\end{equation}
Note that in thisß approach for choosing the LoS, no consideration is given to the cost consequences of the decision since the main focus is to optimize the gain in accrued knowledge about $H_0$ and $H_1$ from the result of the decision function.

\subsection{Sample Size Determination}
\label{subsec-sample size determination}

An important issue that arises is that of determining the proper sample size $n$ to use in a given study. The ROC function $\rho(\cdot)$ depends on such a sample size, so that we may write $\rho(\alpha;n)$ and $\rho^\prime(\alpha;n)$. In each of the approaches for determining the LoS to utilize, the optimal $\alpha$ will then also depend on $n$, so we may write this as $\alpha^*(n)$. We could then determine the appropriate sample size $n^*$ by specifying a lower bound $b > 0$ for the logarithm of the odds-ratio and letting  
\begin{eqnarray}
n^* = n^*(b)  & = & \min\left\{n \ge 1: \log\left[\frac{\rho(\alpha^*(n);n)}{\alpha^*(n)} \frac{1-\alpha^*(n)}{1 - \rho(\alpha^*(n);n)}\right] \ge b\right\} \nonumber \\
& = & \min\left\{n \ge 1:   \left(\frac{\rho(\alpha^*(n);n)}{1 - \rho(\alpha^*(n);n)}\right) \ge 
\left(\frac{\alpha^*(n)}{1-\alpha^*(n)}\right) e^b\right\}.
\label{sample size}
\end{eqnarray}
This sample size determination approach is in contrast to existing procedures where an LoS value is specified, usually to be $\alpha = .05$, and a lower threshold for the power is also specified, say 0.95, and then the desired sample size is determined.

\subsection{Concrete Example}
\label{subsec-concrete example for LoS choice}

We demonstrate the ideas in this section by using a simple normal model. Let $X_1, X_2, \ldots, X_n$ be IID $N(\mu,\sigma^2)$ with $\sigma^2$ known, and consider the problem of deciding between $H_0: \mu = \mu_0$ versus $H_1: \mu = \mu_1$ with $\mu_1 > \mu_0$. Since the sample mean $\bar{X}$ is sufficient for $\mu$, we reduce the problem to just having $\bar{X} \sim N(\mu,\sigma^2/n)$. The MP decision function of size $\alpha$ is
$$\delta(\bar{X};\alpha) = I\{\bar{X} > \mu_0 + \Phi^{-1}(1-\alpha) (\sigma/\sqrt{n})\}$$
whose associated ROC function is, with $\xi = (\mu_1-\mu_0)/\sigma$ being the standardized effect size,
\begin{equation}
\label{rho in concrete example 2}
\rho(\alpha) = 1 - \Phi\left[\Phi^{-1}(1-\alpha) - \xi\sqrt{n}\right].
\end{equation}
Its derivative is
\begin{equation}
\label{deriv concrete example 2}
\rho^\prime(\alpha) = \frac{\phi\left[\Phi^{-1}(1-\alpha) - \xi\sqrt{n}\right]}{\phi\left[\Phi^{-1}(1-\alpha)\right]} = \exp\left\{\xi\sqrt{n}[\Phi^{-1}(1-\alpha) - \xi\sqrt{n}/2]\right\}.
\end{equation}
We now determine $\alpha_M$, $\alpha_B$, and $\alpha_D$ according to the prescriptions in subsections \ref{subsec-minimax}, \ref{subsec-Bayes}, and \ref{subsec-change}, respectively.

To determine $\alpha_M$ we need to solve the equation
$$\rho(\alpha) + R_1 \alpha - R_0 = 0$$
with $R_1 = (C_{01} - C_{00})/(C_{10} - C_{11})$ and $R_0 = (C_{10} - C_{00})/(C_{10} - C_{11})$. Letting $w = \Phi^{-1}(1 - \alpha)$, so that $\alpha = 1 - \Phi(w)$, the resulting equation to be solved with respect to $w$ is
$$\Phi(w - \xi\sqrt{n}) + R_1 \Phi(w) - R_2 = 0$$
with $R_2 = R_1 +1 - R_0 = (C_{01} - C_{11})/(C_{10} - C_{11})$. This equation could be solved numerically, e.g., via Newton-Raphson algorithm, to obtain $w_M$, from which we then obtain $\alpha_M = 1 - \Phi(w_M)$. In the special case where $R_1 = 1$ and $R_2 = 1$, which arises when $C_{01} = C_{10} > 0$ and $C_{00} = C_{11} = 0$, we directly obtain $w_M = \xi\sqrt{n}/2$ leading to
$$\alpha_M = \Phi\left[\frac{-\xi\sqrt{n}}{2}\right] \quad \mbox{and} \quad \rho(\alpha_M) = \Phi\left[\frac{\xi\sqrt{n}}{2}\right].$$

To determine $\alpha_B$, we need to solve the equation
$$\rho^\prime(\alpha) = R_3 \equiv \frac{\kappa_0}{1-\kappa_0} \frac{C_{01}-C_{00}}{C_{10} - C_{11}}.$$
Using the expression in (\ref{deriv concrete example 2}), and solving directly for $\alpha_B$, we find that
\begin{eqnarray*}
\alpha_B & = & \Phi\left[-\frac{\xi\sqrt{n}}{2} - \frac{1}{\xi\sqrt{n}} \log(R_3)\right] \quad \mbox{and} \quad
\rho(\alpha_B)  =  \Phi\left[\frac{\xi\sqrt{n}}{2} - \frac{1}{\xi\sqrt{n}} \log(R_3)\right].
\end{eqnarray*}
In the special case when $R_3 = 1$, such as when $\kappa_0 = 1/2$, the least favorable prior, and with $C_{01} = C_{10} > 0$ and $C_{00} = C_{11} = 0$, we obtain
$$\alpha_B  =  \Phi\left[-\frac{\xi\sqrt{n}}{2}\right] \quad \mbox{and} \quad
\rho(\alpha_B)  =  \Phi\left[\frac{\xi\sqrt{n}}{2}\right],$$
thus coinciding with the special case for $\alpha_M$.

For $\alpha_D$, we need to find the maximizer of the mapping $\alpha \mapsto H(\alpha)$ with
$$H(\alpha) = [\rho(\alpha) - \alpha] \log\left[\frac{\rho(\alpha)}{\alpha} \frac{1-\alpha}{1-\rho(\alpha)}\right].$$
Observe that $H(\alpha) > 0$ for $\alpha \in (0,1)$; $\lim_{\alpha \downarrow 0} H(\alpha) = 0$, and $\lim_{\alpha \uparrow 1} H(\alpha) = 0$ by Proposition \ref{prop-ROC properties}. Differentiating $H(\alpha)$, we obtain
$$H^\prime(\alpha) = [\rho^\prime(\alpha) - 1] \log\log\left[\frac{\rho(\alpha)}{\alpha} \frac{1-\alpha}{1-\rho(\alpha)}\right] + [\rho(\alpha) - \alpha] \left[
\frac{\rho^\prime(\alpha)}{\rho(\alpha)(1-\rho(\alpha))} - \frac{1}{\alpha(1-\alpha)}\right].$$
From the special case when solving for $\alpha_B$, we have $\rho^\prime(\alpha) - 1 = 0$ satisfied by $\alpha^* = \Phi(-\xi\sqrt{n}/2)$, and for this $\alpha^*$ we have $\rho(\alpha^*) =  1 - \alpha^*$ and $\alpha^* = 1 - \rho(\alpha^*)$, so that $H^\prime(\alpha^*) = 0$. Furthermore, since $H^\prime(\alpha)$ goes from positive values to negative values as $\alpha$ goes from 0 to 1, then in fact $\alpha^* = \Phi(-\xi\sqrt{n}/2)$ is the maximizer of $\alpha \mapsto H(\alpha)$. Therefore, we have shown that
$$\alpha_D = \alpha^* =  \Phi\left[\frac{-\xi\sqrt{n}}{2}\right] \quad \mbox{and} \quad \rho(\alpha_D)  = \Phi\left[\frac{\xi\sqrt{n}}{2}\right].$$
This solution coincides with the solutions for $\alpha_M$ and $\alpha_B$ under the special cases arising from $\kappa_0 = 1/2, C_{01} = C_{10} > 0$, and $C_{00} = C_{11} = 0$. Note that the $\alpha_D$ does not depend on $\kappa_0$ and the $C_{ij}$'s since its derivation does not take into account the decision-maker's cost consequences of his decision and his prior knowledge about $H_0$.

To be able to concretely compare the values of $\alpha_M$, $\alpha_B$, and $\alpha_D$, we obtain them under this normal one-sample setting and for different combinations of values of $C_{01}$, $C_{10}$, $\kappa_0$, $n$, and $\xi = (\mu_1-\mu_0)/\sigma$, the effect size. We set $C_{00} = C_{11} = 0$. The computation of $\alpha_M$ utilized the Newton-Raphson algorithm. Table \ref{table-values of alphas} summarizes the values obtained under the different scenarios. From this table, observe that as the effect size increases, the LoS values from each of the three approaches decrease. Furthermore, note that when $\xi = .5$ and under the case where $C_{01} = C_{10} = 1, \kappa_0 = .5$, the LoS values are much larger than $\alpha = .05$. The reason for this is that with these higher LoS values, more information accrues based on the outcome $d$ of the decision function compared to when $\alpha = .05$. In addition, notice the drastic change in $\alpha_M$ and $\alpha_B$ when the cost consequences are not equal compared to when they were equal. In essence, the choice of LoS should depend on the other elements of the decision problem, hence a fixed threshold of $\alpha = .05$, for whatever decision problem at hand, is just untenable and unjustifiable.

{
\begin{table}
\caption{LoS values $\alpha_M$, $\alpha_B$, and $\alpha_D$ associated with the different approaches for determining an optimal LoS for the one-sample normal model under different combinations of costs $C_{01}$, $C_{10}$, $H_0$ prior probability $\kappa_0$, sample size $n$, and effect size $\xi = (\mu_1-\mu_0)/\sigma$.}
\label{table-values of alphas}
\begin{center}
\begin{tabular}{|c|ccccc|ccc|} \hline
Setting \# & $C_{01}$ & $C_{10}$ & $\kappa_0$ & $n$ & $\xi$ & $\alpha_M$ & $\alpha_B$ & $\alpha_D$ \\ \hline
1        & 1 & 1 & 0.5 & 1 & 0.5 & 0.401294 & 0.401294 & 0.401294 \\
2          & 1 & 1 & 0.5 & 1 & 1 & 0.308538 & 0.308538 & 0.308538 \\
3    & 1 & 1 & 0.5 & 1 & 2 & 0.158655 & 0.158655 & 0.158655 \\ \hline
4        & 1 & 1 & 0.5 & 5 & 0.5 & 0.288075 & 0.288075 & 0.288075 \\
5          & 1 & 1 & 0.5 & 5 & 1 & 0.131776 & 0.131776 & 0.131776 \\
6    & 1 & 1 & 0.5 & 5 & 2 & 0.012674 & 0.012674 & 0.012674 \\ \hline
7       & 1 & 1 & 0.25 & 1 & 0.5 & 0.401294 & 0.974246 & 0.401294 \\
8         & 1 & 1 & 0.25 & 1 & 1 & 0.308538 & 0.725284 & 0.308538 \\
9   & 1 & 1 & 0.25 & 1 & 2 & 0.158655 & 0.326105 & 0.158655 \\ \hline
10      & 1 & 1 & 0.25 & 5 & 0.5 & 0.288075 & 0.664075 & 0.288075 \\
11        & 1 & 1 & 0.25 & 5 & 1 & 0.131776 & 0.265422 & 0.131776 \\
12  & 1 & 1 & 0.25 & 5 & 2 & 0.012674 & 0.023273 & 0.012674 \\ \hline
13         & 10 & 1 & 0.5 & 1 & 0.5 & 0.081468 & 1e-06 & 0.401294 \\
14        & 10 & 1 & 0.5 & 1 & 1 & 0.068649 & 0.002535 & 0.308538 \\
15  & 10 & 1 & 0.5 & 1 & 2 & 0.040108 & 0.015727 & 0.158655 \\ \hline
16      & 10 & 1 & 0.5 & 5 & 0.5 & 0.065308 & 0.004416 & 0.288075 \\
17        & 10 & 1 & 0.5 & 5 & 1 & 0.034035 & 0.015866 & 0.131776 \\
18  & 10 & 1 & 0.5 & 5 & 2 & 0.003666 & 0.002971 & 0.012674 \\ \hline
19     & 10 & 1 & 0.25 & 1 & 0.5 & 0.081468 & 0.003931 & 0.401294 \\
20       & 10 & 1 & 0.25 & 1 & 1 & 0.068649 & 0.044193 & 0.308538 \\
21 & 10 & 1 & 0.25 & 1 & 2 & 0.040108 & 0.054579 & 0.158655 \\ \hline
22     & 10 & 1 & 0.25 & 5 & 0.5 & 0.065308 & 0.050932 & 0.288075 \\
23       & 10 & 1 & 0.25 & 5 & 1 & 0.034035 & 0.048814 & 0.131776 \\
24 & 10 & 1 & 0.25 & 5 & 2 & 0.003666 & 0.006118 & 0.012674 \\ \hline
\end{tabular}
\end{center}
\end{table}
}

Depending on which rationale the decision-maker uses for his determination of his LoS $\alpha^*$, he could then determine an appropriate sample size for his experiment or study according to the formula (\ref{sample size}) in subsection \ref {subsec-sample size determination}. We demonstrate this for the case of $\alpha_D$, which also happens to be the same solutions for $\alpha_M$ and $\alpha_B$ in special case where $C_{01} = C_{10} = 1$ and $\kappa_0 = .5$. Given a lower bound $b > 0$, from (\ref{sample size}), we seek the smallest integer $n$ satisfying
$$\left[\frac{\Phi(\xi\sqrt{n}/2)}{1 - \Phi(\xi\sqrt{n}/2)}\right] \ge \left[\frac{1 - \Phi(\xi\sqrt{n}/2)}{\Phi(\xi\sqrt{n}/2)}\right] e^b.$$
Directly solving this inequality, the desired sample size $n^* = n^*(b)$ is the smallest integer at least equal to
$$\bar{n} = \bar{n}(b) = \frac{4 \left[\Phi^{-1}\left(\frac{\exp(b/2)}{1 + \exp(b/2)}\right)\right]^2 \sigma^2}{(\mu_1-\mu_0)^2}.$$
For this sample size, determined by specifying the lower bound $b$, we find that
$$\alpha_D(b) = \frac{1}{1 + \exp(b/2)} \quad \mbox{and} \quad 
\rho(\alpha_D(b)) = \frac{\exp(b/2)}{1+\exp(b/2)}.$$
When $b=6$, this becomes
$$\bar{n}(b=6) = \frac{4 (1.6703)^2 \sigma^2}{(\mu_1-\mu_0)^2}.$$
For this $b = 6$ and $\bar{n}(b=6)$, we obtain $\alpha_D(b=6) = 0.0474$ and $\rho(\alpha_D(b=6)) = 0.9526$.
Recall that in the usual sample size calculation where we specify the LoS to be $\alpha = .05$ and with the power to be $1 - \beta = .95$, the required sample size is the smallest integer at least equal to
$$\tilde{n} = \frac{[\Phi^{-1}(1-\beta) + \Phi^{-1}(1-\alpha)]^2 \sigma^2}{(\mu_1-\mu_0)^2} =
\frac{4 (1.645)^2 \sigma^2}{(\mu_1 - \mu_0)^2}.$$
Interestingly, if one wants {\em the} $b$ that will lead to an LoS of exactly $\alpha = .05$, we find that you should take $b = 2 \log(95/5) = 5.888\bar{9}$, which is a rather mysterious number (at least compared to the perfect number\footnote{A perfect number is such that it is equal to the sum of its proper divisors. The number 6 is the smallest perfect number since it is the sum of its proper divisors 1, 2, and 3.} $b = 6$), lending further mystery to the common, typical, conventional, and age-old choice of an LoS of $\alpha = 0.05$!

\section{Summary and Concluding Remarks}
\label{sec-Conclusions}

Humankind's desire to acquire knowledge in the pursuit of the truth of whatever phenomenon is at hand, be it about ``eternal and long-lasting truths'' according to Efron, or of matters more ephemeral and temporally of passing interest, such as President Trump's income taxes, is deeply ingrained in our DNA. The search for truth has especially  both fascinated and confounded us all in the past few years, what with Trump's adviser Rudy Giuliani exclaiming on national television about ``Truth isn't Truth'' or with the proliferation of fake news for nefarious purposes which are unwittingly being facilitated by technological platforms in the social media arena. In fact, even in books and fictional novels, the search for truth is of paramount interest. Indeed, from the best-selling fictional novel titled {\em The Art of Racing in the Rain} by Garth Stein \cite{GarthStein18}, which has been adapted for film by Walt Disney Studios Motion Pictures and released this year, we encounter a dialogue involving the main canine character, Enzo:

\begin{quote}
``Inside each of us resides the truth,'' I began, ``the absolute truth. But sometimes the truth is hidden in a hall of mirrors. Sometimes we believe we are viewing the real thing, when in fact we are viewing a facsimile, a distortion. As I listen to this trial, I am reminded of the climactic scene of a James Bond film, {\em The Man with the Golden Gun}. James Bond escaped his hall of mirrors by breaking the glass, shattering the illusions, until only the true villain stood before him. We, too, must shatter the mirrors. We must look into ourselves and root out the distortions until that thing which we know in our hearts is perfect and true, stands before us. Only then will justice be served.''
\end{quote}

The development of appropriate methods useful in the search for truth and the acquisition of knowledge using data is one of the foremost, if not the main, goal of statisticians. We, statisticians and data scientists, must continue re-examining existing methods and refine them if there are problems, so that we may facilitate the `breaking of the glass and the shattering of the illusions'. Controversies have arisen in the use of existing statistical procedures for decision-making, especially those that utilize $P$-values and the setting of hard thresholds, such as $\alpha = .05$, for the level of significance.  Thus, there is a need to re-examine these existing procedures to determine if they could be improved or altogether replaced by newer methods. This paper is in this spirit. It re-examined existing statistical decision-making approaches in the most fundamental of settings, that of deciding between, or sequentially acquiring knowledge about, two simple hypotheses. It is expected that the clarity provided by dealing with this fundamental setting will result in improving methods for more complicated settings, in a similar vein that the Neyman-Pearson Fundamental Lemma \cite{NeyPea33} enabled the solutions of hypothesis testing problems in complicated settings.

The Neyman-Pearson MP decision function for deciding between two simple hypotheses is still the linchpin of the theory of statistical decision-making. The classical approach is to specify the LoS $\alpha$, which is the maximum allowable probability of a Type I error. One then looks for that decision function that maximizes the power under $H_1$, the MP decision function. It turns out, however, that the more beneficial approach is to consider the stochastic process of MP decision functions indexed by $\alpha$ and to look at the receiver operating characteristic function of the decision process, which is the power of the MP decision function as a function of $\alpha$. We then clarified the notion of the $P$-functional, usually called the $P$-value statistic, as a functional of the decision process.
The notion of replicability was then examined in the context of observing the outcomes of a decision function and the $P$-functional. We concur with the NAS report \cite{NASReport2019} that replicability should be viewed in `the context of an entire body of evidence, rather than an individual study or an individual replication.' This led to the question on how the results of studies should be reported. We argued that when reporting the outcome of the decision function, that it is imperative to accompany it with the value of either the likelihood ratio or the logarithm of the likelihood ratio based on observing the decision function. This value will provide a measure of the strength of support, either for $H_0$ or $H_1$, given the decision. Whatever LoS value is used will be automatically incorporated into this summary measure, so that a decision-maker is actually free to choose whatever LoS value he desires, so long as it is not decided after seeing the data. 

If one wants to report the value $p$ of the $P$-functional, this should {\em simply} be used to make the decision, which is to reject $H_0$ if and only if $p$ is no more than the specified LoS $\alpha$, which in fact corresponds to the decision made using the level $\alpha$ MP decision function. The value itself of $p$ could be misleading in terms of the information it provides about whether the results of the study supports $H_0$ or $H_1$. For instance, it is possible that $p = 0.0001$, but this could still be more supportive of $H_0$ than of $H_1$. Rather, what should be reported is the value of the $P$-functional density under $H_1$, given by $\rho^\prime(p)$, or equivalently $\log \rho^\prime(p)$, which is actually the likelihood ratio or the log-likelihood ratio, respectively, based on observing the $P$-functional. This quantity is more informative compared to just the value of $p$. Furthermore, log-likelihood ratio profile curves or contour plots could also be provided, especially when dealing with composite alternative hypothesis.

Discussion was then presented on updating knowledge of $H_0$ and $H_1$ when given the realization $x$ of the random observable $X$, the value $d$ of a decision function $\delta^*(\alpha)$, or the value $p$ of the $P$-functional, via Bayes theorem. It is demonstrated that a coherent way of acquiring knowledge about $H_0$ and $H_1$ is through this process of sequential learning, where the current states of knowledge of $H_0$ and $H_1$ are updated, using Bayes theorem, when results from new studies are reported. It is established that if there is no publication or reporting bias, then the sequence of posterior probabilities will converge to 1 (0) if $H_0$ is true (false). Through sequential learning, any study performed with integrity and honesty, whatever its results, will contribute to the acquisition of knowledge about $H_0$ and $H_1$. This is the ideal situation of a community of researchers, both cooperatively and competitively, seeking the truth about a certain phenomenon.  However, it was also demonstrated that in the presence of publication or reporting bias, a monkey-wrench is thrown into the knowledge updating that, {\em even under $H_0$}, it is possible to have the sequence of posterior probabilities of $H_0$ to converge to 0, instead of 1.

It was argued that a decision-maker is free to choose any LoS value that he desires since it will automatically be accounted for by the likelihood ratio based on the realized decision which should accompany the decision. In essence, decision-makers should not consider an LoS value of $\alpha = .05$ as having been `carved on stone tablets.'  In fact, insisting on $\alpha = .05$ which is a value not justifiable within the theory of statistical decision-making, renders the mathematical theory `impure', and this has led to controversies and criticisms.  The question therefore arose on how a decision-maker could optimally choose his LoS. Three approaches were discussed for determining such an optimal LoS, each depending on the viewpoint that the decision-maker is operating on. These approaches led to LoS values which are justifiable within the theory, thereby making the theory `beautiful' and free of arbitrary inputs. With these approaches to deciding on the LoS, a new procedure for sample size determination was also developed. 
The ideas in this paper were demonstrated using concrete examples pertaining to a two-sample problem, a one-sample problem, and the famous Fisher's lady tea-tasting experiment. It is hoped that the ideas presented and the proposed changes could improve existing statistical decision-making methods and hopefully eliminate, or at least lessen, the criticisms being hurled against these existing statistical decision-making methods. 

\section*{Acknowledgements}

This research was partially supported by NIH Grant P30GM103336-01A1 and the 
Center for Colon Cancer Research at the University of South Carolina (UofSC). The author acknowledges his regular walking companion, Albert E.!, since many of the ideas arose during their walks.

\bibliography{SearchForTruth}
\bibliographystyle{plain}

\end{document}